\documentclass[11pt]{article}
\usepackage{amsmath,amsfonts,amsthm,hyperref,enumerate,graphicx}
\usepackage[left=1in,top=1in,right=1in]{geometry}
\usepackage{color}

\renewcommand{\l}{\left}
\renewcommand{\r}{\right}
\theoremstyle{definition}
\newtheorem{theorem}{Theorem}[section]
\newtheorem{lemma}[theorem]{Lemma}
\newtheorem{proposition}[theorem]{Proposition}
\newtheorem{claim}[theorem]{Claim}
\newtheorem{problem}[theorem]{Problem}
\counterwithin{equation}{section}
\allowdisplaybreaks

\begin{document}

\title{\textbf{Random Tur{\'a}n and counting results for general position sets over finite fields}}

\author{Yaobin Chen\thanks{Shanghai Center for Mathematical Sciences, Fudan University, Shanghai, 200438 China. Email: {\tt ybchen21@m.fudan.edu.cn}. Supported by National Natural Science Foundation of China (Grant No. 123B2012).}
\and Xizhi Liu\thanks{Mathematics Institute and DIMAP, University of Warwick, Coventry, CV4 7AL, UK. Email: {\tt xizhi.liu.ac@gmail.com}. Supported by ERC Advanced Grant 101020255 and Leverhulme Research Project Grant RPG-2018-424.}
\and Jiaxi Nie\thanks{Shanghai Center for Mathematical Sciences, Fudan University, Shanghai, 200438 China. Email: {\tt jiaxi\_nie@fudan.edu.cn}.}
\and Ji Zeng\thanks{University of California San Diego, La Jolla, CA, USA and Alfréd Rényi Institute of Mathematics, Budapest, Hungary. Email: {\tt jzeng@ucsd.edu}. Supported by NSF grant DMS-1800746, and ERC Advanced Grants ``GeoScape'', no. 882971 and ``ERMiD'', no. 101054936.}}
\date{}

\maketitle

\begin{abstract}
Let $\alpha(\mathbb{F}_q^d,p)$ denote the maximum size of a general position set in a $p$-random subset of $\mathbb{F}_q^d$. We determine the order of magnitude of $\alpha(\mathbb{F}_q^2,p)$ up to polylogarithmic factors for all possible values of $p$, improving the previous results obtained by Roche-Newton--Warren and Bhowmick--Roche-Newton. For $d \ge 3$ we prove upper bounds for $\alpha(\mathbb{F}_q^d,p)$ that are essentially tight within certain ranges for $p$. 

We establish the upper bound $2^{(1+o(1))q}$ for the number of general position sets in $\mathbb{F}_q^d$, which matches the trivial lower bound $2^{q}$ asymptotically in the exponent. We also refine this counting result by proving an asymptotically tight (in the exponent) upper bound for the number of general position sets with a fixed size. The latter result for $d=2$ improves a result of Roche-Newton--Warren.

Our proofs are grounded in the hypergraph container method, and additionally, for $d=2$ we also leverage the pseudorandomness of the point-line incidence graph of $\mathbb{F}_{q}^2$.
\end{abstract}


\section{Introduction}
Let $d$ be a positive integer and $\mathbb{F}$ be a field. A point set $P$ in $\mathbb{F}^{d}$ is in \textit{general position} if no $d+1$ points of $P$ are contained in a hyperplane, where a \textit{hyperplane} is a $(d-1)$-dimensional affine subspace of $\mathbb{F}^{d}$. There are several central problems closely related to general position sets in extremal combinatorics and discrete geometry. For example, the famous \textit{no-three-in-line problem} raised by Dudeney \cite{Dudeney} in 1917 asks if there exists a general position set in $\mathbb{R}^2$ that contains $2n$ points from the grid $[n]\times [n]$. This question and its variations have been extensively studied for a long time, see \cite{hall1975some,flammenkamp1998progress,por2007no,lefmann2008no,SZ} for some latest results. A closely related problem posted by Erd\H{o}s~\cite{erdos1986some} asks for the maximum size of a general position set contained in an arbitrary point set of size $n$ without collinear quadruples in $\mathbb{R}^2$. In 2018, a breakthrough result by Balogh and Solymosi \cite{BS} using the hypergraph container method shows that $n^{5/6+o(1)}$ is an upper bound for Erd{\H o}s' question.

In the present work, we focus on extremal problems related to general position sets in $\mathbb{F}^d_q$, where $\mathbb{F}_{q}$ is the finite field whose cardinality is the prime power $q$. More specifically, we study the maximum size of a general position set contained in a $p$-random subset of $\mathbb{F}_{q}^{d}$ and the number of general position sets in $\mathbb{F}_{q}^{d}$. Our results are presented in the following two subsections. 

\subsection{Random Tur{\' a}n results}

The random Tur\'an-type problem is a well-studied branch in extremal and probabilistic combinatorics concerning the maximum value of certain parameters in random structures. For example, see~\cite{kohayakawa1997k,conlon2016combinatorial,schacht2016extremal,jiang2022balanced,spiro2022random,Mubayi2023OnTR,nie2023tur,nie2023random} for some results of this flavor. In this subsection, we study random Tur\'an-type problems for general position sets in $\mathbb{F}_{q}^{d}$.

Given a real number $p\in [0,1]$, we use $\alpha(\mathbb{F}_q^d,p)$ to denote the maximum size of a general position set that is contained in a \textit{$p$-random set} $\mathbf{S}_p \subset \mathbb{F}_q^d$. Here, by ``$p$-random'' we mean that every point of $\mathbb{F}_q^d$ is sampled into $\mathbf{S}_p$ independently with probability $p$.

Note that determining $\alpha(\mathbb{F}_q^d,1)$ is equivalent to answering the following fundamental question: How large can a general position set in $\mathbb{F}^d_q$ be? Erd\H{o}s~\cite{roth1951problem} observed that the \textit{moment curve}, which consists of all points $(x, x^2, \dots, x^d) \in \mathbb{F}_{q}^d$, is in general position. Hence we have the lower bound $\alpha(\mathbb{F}_q^d,1) \ge q$. On the other hand, since $\mathbb{F}_q^d$ is a disjoint union of $q$ hyperplanes and a general position set can contain at most $d$ points from each of them, we obtain the upper bound $\alpha(\mathbb{F}_q^d,1) \le dq$. In fact, it follows from a more involved argument that $\alpha(\mathbb{F}_q^d,1)\le(1+o(1))q$, see e.g. our Lemma~\ref{supersaturation_Fqd}. Overall, we have $\alpha(\mathbb{F}_q^d,1) = (1+o(1))q$.

For $d=2$ and $p \neq 1$, Roche-Newton--Warren~\cite{roche2022arcs} and Bhowmick--Roche-Newton~\cite{bhowmick2022counting} established essentially tight bounds for $\alpha(\mathbb{F}_q^2,p)$ when $p$ lies in the intervals $[0,q^{-1}]$ and $[q^{-1/3},1]$ respectively. We improve their results by determining the order of magnitude of $\alpha(\mathbb{F}_q^2,p)$ up to polylogarithmic factors for all possible values of $p$ (see Figure~\ref{fig:randomturan1}).

\begin{figure}[ht]
    \centering
    \includegraphics[scale=0.7, trim=180 150 250 136,clip]{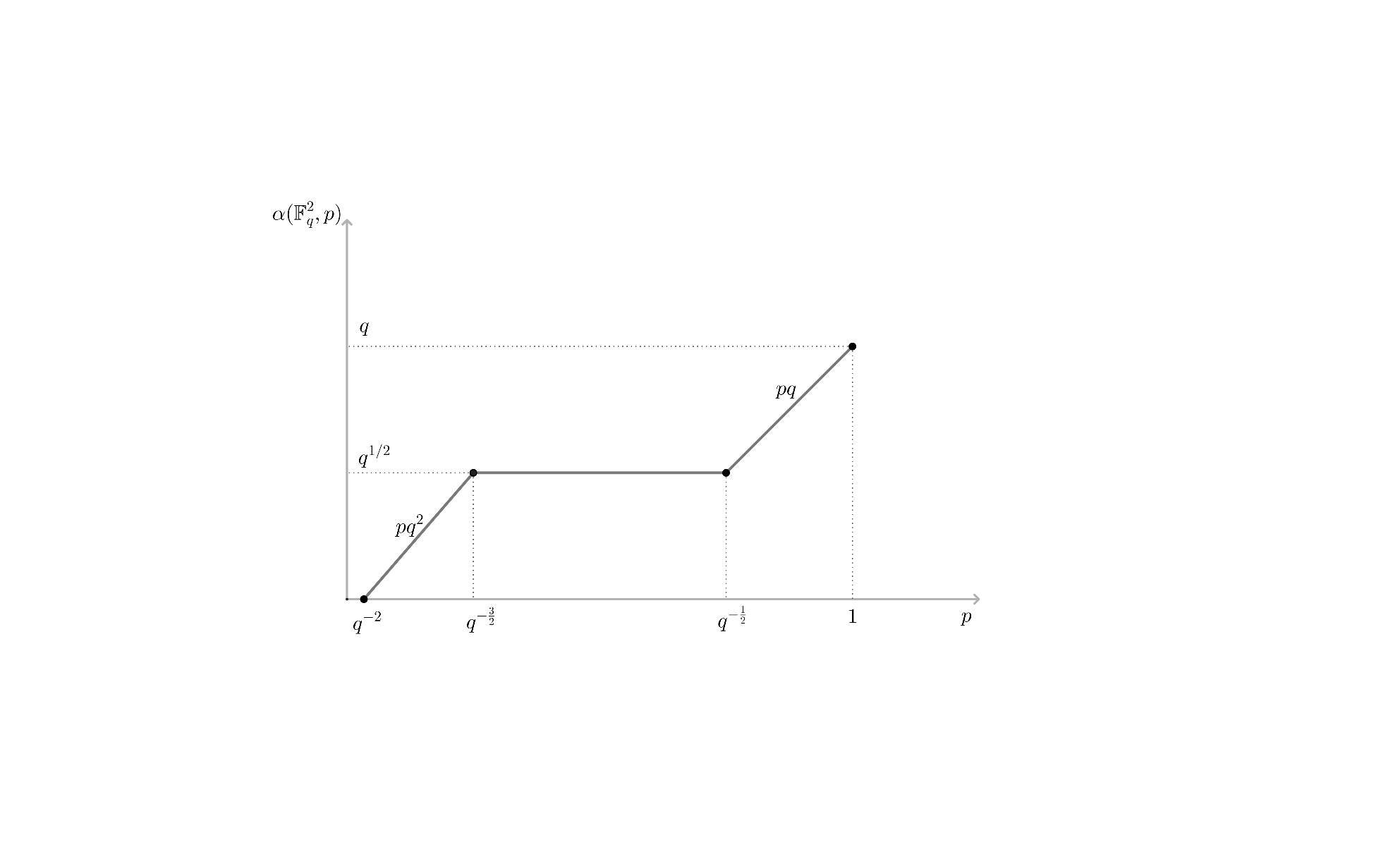}
    \caption{The behaviour of $\alpha(\mathbb{F}_q^2,p)$ in terms of $p$.}
    \label{fig:randomturan1}
\end{figure}

\begin{theorem}\label{thm:randomturan1}
As the prime power $q\to \infty$, asymptotically almost surely, we have 
\begin{equation*}
\alpha(\mathbb{F}_q^2,p)=
\l\{
    \begin{aligned}
    &\Theta(pq^{2}),~~~&q^{-2+o(1)}\le p\le q^{-3/2-o(1)},\\
    &q^{1/2+o(1)},~~~&q^{-3/2-o(1)}\le p\le q^{-1/2+o(1)},\\
    &\Theta(pq),~~~&q^{-1/2+o(1)}\le p\le 1.
\end{aligned}
\r.
\end{equation*}
Moreover, all $q^{o(1)}$ factors here are polylogarithmic.
\end{theorem}

Similar to the previous works \cite{roche2022arcs,bhowmick2022counting}, our proof is grounded in the hypergraph container method. The novelty in our argument is a balanced supersaturation result for collinear triples in reasonably large subsets of $\mathbb{F}_q^2$, where the ``balanced'' property is achieved by leveraging the pseudorandomness of the point-line incidence graph of $\mathbb{F}_{q}^2$.

In higher dimensions, there are new difficulties arising in determining $\alpha(\mathbb{F}_q^d,p)$ for all $p$ (see our remark in Section~\ref{sec:remark}). Nevertheless, we extend the results in \cite{roche2022arcs,bhowmick2022counting} to every dimension $d\ge 3$ by establishing essentially tight bounds for $\alpha(\mathbb{F}_q^d,p)$ when $p$ lies in certain intervals (see Figure~\ref{fig:randomturan2}). Our proof utilizes a notion of ``critical coplanarity'' (see Lemma~\ref{critical_supersaturation_Fqd}) to partially overcome the technical challenges that arise only in higher dimensions.

\begin{figure}[ht]
    \centering
    \includegraphics[scale=0.6,trim=60 130 200 350,clip]{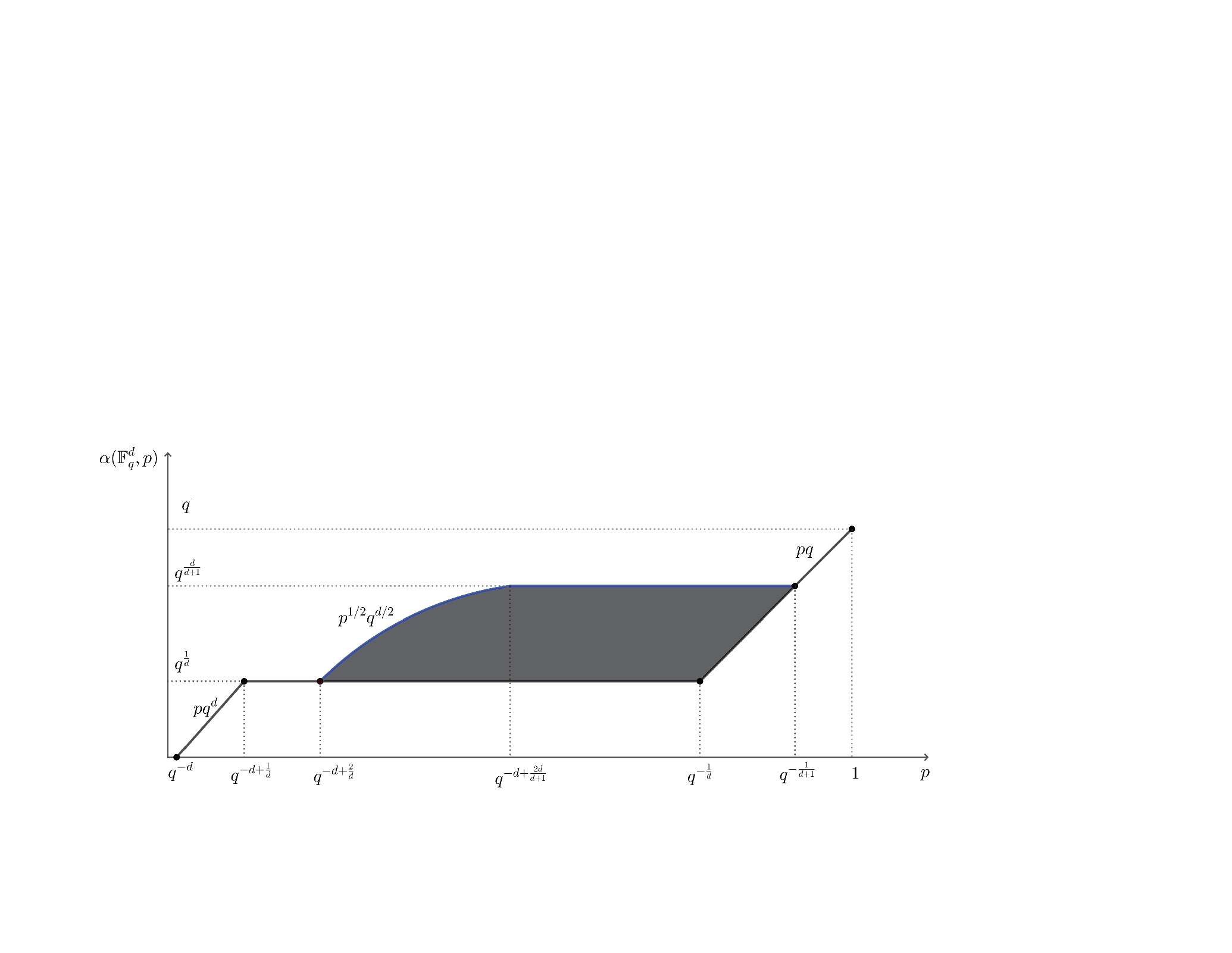}
    \caption{Lower and upper bounds for $\alpha(\mathbb{F}_q^d,p)$ in terms of $p$.}
    \label{fig:randomturan2}
\end{figure}

\begin{theorem}\label{thm:randomturan2}
Let $d\ge 3$ be a fixed integer. As the prime power $q\to \infty$, asymptotically almost surely, we have 
\begin{equation*}
\alpha(\mathbb{F}_q^d,p)=
\l\{
    \begin{aligned}
    &\Theta(pq^d),~~&q^{-d+o(1)} \le p\le q^{-d+\frac{1}{d}-o(1)},\\
    &q^{\frac{1}{d}+o(1)},~~&q^{-d+\frac{1}{d}-o(1)}\le p\le q^{-d+\frac{2}{d}+o(1)},\\
    &\Theta(pq),~~&q^{-\frac{1}{d+1}+o(1)}\le p\le 1.
\end{aligned}
\r.
\end{equation*}
In addition, when $q^{-d+\frac{2}{d}+o(1)}\le p\le q^{-\frac{1}{d+1}+o(1)}$, we have
\begin{equation*}
    \max\l\{q^{\frac{1}{d}-o(1)},~\Omega(pq)\r\}\le \alpha(\mathbb{F}^d_q,p)\le \min\l\{p^{\frac{1}{2}}q^{\frac{d}{2}+o(1)},~q^{\frac{d}{d+1}+o(1)}\r\}.
\end{equation*} Moreover, all $q^{o(1)}$ factors here are polylogarithmic.
\end{theorem}

Note that Theorem~\ref{thm:randomturan2} determines the asymptotic value of $\log(\alpha(\mathbb{F}_q^d,p))$ when $p$ lies in the intervals $[q^{-d}, q^{-d+2/d+o(1)}]$ and $[q^{-1/(d+1)+o(1)}, 1]$, which leaves a gap between the lower and upper bounds when $p$ lies in the interval $[q^{-d+2/d+o(1)}, q^{-1/(d+1)+o(1)}]$. It is an interesting problem to close this gap.

\subsection{Counting results}

In this subsection, we focus on the following two counting problems: How many general position sets does $\mathbb{F}^d_q$ contain? And how many of them are of a fixed size $m$? Using the moment curve defined in the previous subsection, one can easily obtain the lower bounds $2^{q}$ and $\binom{q}{m}$ to these questions respectively.

For the upper bound, Bhowmick--Roche-Newton \cite{bhowmick2022counting} proved that $\mathbb{F}_q^2$ contains at most $2^{(1+o(1))q}$ general position sets. Note that the exponents of their upper bound and the lower bound $2^q$ match asymptotically. The following theorem extends this result to all dimensions.

\begin{theorem}\label{thm:counting1}
    Let $d\ge 2$ be a fixed integer and the prime power $q$ tend to infinity. The number of general position sets in $\mathbb{F}_q^d$ is at most $2^{q+q^{\frac{2d}{2d+1}+o(1)}}$.
\end{theorem}

For counting general position sets with a fixed size, Bhowmick--Roche-Newton \cite{bhowmick2022counting} proved that $\mathbb{F}_q^2$ contains at most $\binom{(1+o(1))q}{m}$ general position sets of size $m$ for $m \in [q^{2/3+o(1)}, q]$, and this matches the lower bound $\binom{q}{m}$ asymptotically in the exponent. We also extend this result to all dimensions.

\begin{theorem}\label{thm:counting2}
    Let $d\ge 2$ be a fixed integer and the prime power $q$ tend to infinity. For every integer $m \ge q^{d/(d+1)}(\log q)^3$, the number of general position sets of size $m$ in $\mathbb{F}_q^d$ is at most $\binom{\left(1+o(1)\right)q}{m}$.
\end{theorem}

For $m$ much smaller than $q^{d/(d+1)}$, the upper and lower bounds remain widely far. In this direction, Roche-Newton--Warren \cite{roche2022arcs} proved that $\mathbb{F}_q^2$ contains at most $\binom{q^{2+o(1)}/m}{m}$ many general position sets of size $m$ for $m \geq q^{1/2 + o(1)}$. The following theorem greatly improves this result. 

\begin{theorem}\label{thm:counting3}
    There exists an absolute constant $c> 0$ such that for sufficiently large prime power $q$ and every integer $m\ge q^{1/2}(\log q)^2$, the number of general position sets of size $m$ in $\mathbb{F}_q^2$ is at most $\binom{cq}{m}$.   
\end{theorem}

It is known that $\mathbb{F}_q^2$ contains at least $\binom{q^{2-o(1)}}{m}$ many general position sets of size $m$ for $m \ll q^{1/2}$ (see Theorem~1 of~\cite{roche2022arcs}). Hence, one cannot hope for a polynomial relaxation of the condition ``$m \geq q^{1/2 + o(1)}$'' in Theorem~\ref{thm:counting3}.

Our proofs of Theorems~\ref{thm:counting1} and~\ref{thm:counting2} generalize those in \cite{bhowmick2022counting} while overcoming some technical challenges that arise only in higher dimensions (mostly in upper bounding codegrees during the construction of containers). On the other hand, our argument for Theorem~\ref{thm:counting3} crucially relies on the pseudorandomness of the point-line incidence graph of $\mathbb{F}_{q}^2$ just like Theorem~\ref{thm:randomturan1}.

\bigskip

The rest of this paper is organized as follows. In Section~\ref{sec:preliminary}, we present some useful preliminary results, especially the hypergraph container lemma. In Section~\ref{sec:randomturan1}, we utilize the pseudorandomness of the point-line incidence graph of $\mathbb{F}_{q}^2$ to prove Theorem~\ref{thm:randomturan1}. Section~\ref{sec:counting} is devoted to the proofs of Theorems~\ref{thm:counting1}, \ref{thm:counting2}, and~\ref{thm:counting3}. In Section~\ref{sec:randomturan2}, we use a supersaturation result for critical coplanar tuples to establish Theorem~\ref{thm:randomturan2}. Finally, we include some remarks in Section~\ref{sec:remark}. 

Throughout the paper, the asymptotic notations are always understood with the prime power $q$ tending to infinity. We say that a sequence of events $\textbf{A}_q$ happens \textit{asymptotically almost surely} (a.a.s.) if $\text{Pr}[\textbf{A}_q] \rightarrow 1$ as $q \rightarrow \infty$. We omit floors and ceilings whenever they are not crucial for the sake of clarity in our presentation. The functions $\exp(x)$ and $\log(x)$ are both in base $2$. We only consider unordered tuples of points in $\mathbb{F}_q^d$, hence a ``$k$-tuple'' is equivalent to a point set of size $k$.

\section{Preliminaries}\label{sec:preliminary}

We present some preliminary results in this section. Within our proofs of Theorems~\ref{thm:randomturan1} and~\ref{thm:randomturan2}, we shall make free use of Markov's inequality and Chernoff's bound (see e.g. Theorem 22.6 in~\cite{FK16}). For convenience, we call a $k$-dimensional affine subspace in $\mathbb{F}_{q}^d$ a \textit{$k$-flat}. The following counting result for the number of $k$-flats in $\mathbb{F}_q^d$ will be useful for us.

\begin{lemma}[Goldman--Rota~\cite{goldman1969number}]\label{flat_counting}
    The number of $k$-flats in $\mathbb{F}_q^d$ is $q^{d-k} \binom{d}{k}_q$, and 
    the number of $k$-flats in $\mathbb{F}_q^d$ containing a fixed point is $\binom{d}{k}_q$, where
    \begin{equation*}
        \binom{d}{k}_q := \frac{(1-q^d)(1-q^{d-1})\cdots(1-q^{d-k+1})}{(1-q)(1-q^2)\cdots(1-q^k)} = (1+o(1))q^{k(d-k)}.
    \end{equation*}
\end{lemma}

One of the key tools in our proofs will be the following well-known hypergraph container lemma, proved independently by Balogh--Morris--Samotij~\cite{balogh2015independent} and Saxton--Thomason~\cite{saxton2015hypergraph}.  

Let $r\ge 2$ be an integer. An $r$-graph $\mathcal{H}$ is a collection of $r$-subsets (called edges) of some finite ground set $V$, which is called the vertex set of $\mathcal{H}$. We identify an $r$-graph $\mathcal{H}$ with its edge set and use $V(\mathcal{H})$ to denote its vertex set. A subset of $V(\mathcal{H})$ is said to be an \textit{independent set} if it contains no edges of $\mathcal{H}$. Given a subset $U\subset V(\mathcal{H})$, the \textit{degree} $\delta(U)$ of $U$ is the number of edges in $\mathcal{H}$ containing $U$. For $i\in [r]$, let $\binom{V(\mathcal{H})}{i}$ be the collection of all $i$-subsets of $V(\mathcal{H})$ and 
\begin{equation*}
    \Delta_i(\mathcal{H}) := \max\left\{\delta(U) :~ U\in \binom{V(\mathcal{H})}{i}\right\}, 
\end{equation*}
and for every real number $\tau>0$, define 
\begin{equation*}
    \Delta(\mathcal{H},\tau)
    := \frac{2^{\binom{r}{2}-1}|V(\mathcal{H})|}{r|\mathcal{H}|}\sum_{i=2}^{r}\frac{\Delta_{i}(\mathcal{H})}{2^{\binom{i-1}{2}}\tau^{i-1}}.
\end{equation*}

\begin{lemma}[Balogh--Morris--Samotij~\cite{balogh2015independent},  Saxton--Thomason~\cite{saxton2015hypergraph}]\label{hypergraphcontainer}
    Let $\mathcal{H}$ be an $r$-graph and $\varepsilon, \tau \in (0,1/2)$ be real numbers. 
    Suppose that $\tau<1/(200\cdot r \cdot r!)$ and $\Delta(\mathcal{H},\tau)\leq \varepsilon/(12\cdot r!)$, then there exists a collection $\mathcal{C}$ of subsets (called containers) of $V(\mathcal{H})$ such that
        \begin{enumerate}[(i)]
            \item every independent set of $\mathcal{H}$ is contained in some member of $\mathcal{C}$,
            \item $\log|\mathcal{C}|\leq 1000\cdot r \cdot (r!)^3 \cdot |V(\mathcal{H})| \cdot \tau\cdot \log(1/\varepsilon)\cdot \log(1/\tau)$, and
            \item for every $C\in \mathcal{C}$ the induced subgraph $\mathcal{H}[C]$ has size at most $\varepsilon|\mathcal{H}|$.
        \end{enumerate}
\end{lemma}

\section{Random Tur{\'a}n results in \texorpdfstring{$\mathbb{F}_q^2$}{}}\label{sec:randomturan1}

We aim to prove Theorem~\ref{thm:randomturan1} in this section.

\subsection{Pseudorandomness and balanced supersaturation in \texorpdfstring{$\mathbb{F}_q^2$}{}}

In this subsection, we use the pseudorandomness of the point-line incidence graph of $\mathbb{F}_q^2$ to establish a balanced supersaturation result for collinear triples in large point sets.

Recall that the \textit{adjacency matrix} $A_G$ of an $n$-vertex graph $G$ is the $n \times n$ symmetric matrix where
    \begin{equation*}
        A_G(i,j) := \begin{cases}
               1, & \mbox{if } \{i,j\}\in E(G), \\
               0, & {\rm otherwise.}
            \end{cases}
    \end{equation*} 
We use $\lambda_{1}(G) \geq \dots \geq \lambda_{n}(G)$ to denote the eigenvalues of $A_{G}$. It is well-known that if $G$ is bipartite, then $\lambda_i(G) = -\lambda_{n-i+1}(G)$ holds for all $i \in [n]$. A bipartite graph $G$ with bipartition $V_1 \sqcup V_2$ is said to be \textit{$(\delta_1,\delta_2)$-regular} if the degree $\delta_{G}(v) = \delta_i$ for all $v \in V_i$ and $i \in \{1, 2\}$. A simple linear algebra argument tells us that every $n$-vertex $(\delta_1,\delta_2)$-regular bipartite graph $G$ satisfies $\lambda_1(G) = -\lambda_n(G) = \sqrt{\delta_1\delta_2}$.

The seminal expander mixing lemma~\cite{AC88,Hae95,KS06} relates the edge distribution of a graph with its second largest eigenvalue. One of the key tools in the proof of Theorem~\ref{thm:randomturan1} is the following bipartite extension of the expander mixing lemma (see e.g. Lemma~3.7 in \cite{BLM22}).

\begin{lemma}\label{expander_mixing}
    Suppose that $G$ is a $(\delta_1, \delta_2)$-regular bipartite graph with bipartition $V_1\sqcup V_2$. 
    Then for every $S\subset V_1$ and $T\subset V_2$, the number of edges between $S$ and $T$, denoted by $e(S, T)$, satisfies 
    \begin{equation*}
        \left| e(S,T) - \frac{\delta_2}{|V_1|}|S||T| \right|
        \le \lambda_{2}(G)\sqrt{|S||T|}.
    \end{equation*}
\end{lemma}

Let $\mathbb{H}_{q}^d$ denote the collection of all hyperplanes in $\mathbb{F}_{q}^d$. We define the \textit{point-hyperplane incidence graph} of $\mathbb{F}_{q}^{d}$ as the bipartite graph $I_{q,d}$ with bipartition $\mathbb{F}_{q}^d \sqcup \mathbb{H}_{q}^d$ such that for every $p \in \mathbb{F}_{q}^d$ and $\pi \in \mathbb{H}_q^d$, the pair $\{p, \pi\}$ is an edge if and only if $p \in \pi$.

\begin{lemma}\label{incidence_eigenvalue}
    The bipartite graph $I_{q,d}$ is $\left(\frac{q^d-1}{q-1}, q^{d-1}\right)$-regular with $\lambda_{2}(I_{q,d}) = q^{\frac{d-1}{2}}$.
\end{lemma}
\begin{proof}
    According to Lemma~\ref{flat_counting}, $|\mathbb{H}_{q}^d| = \frac{q(q^{d}-1)}{q-1}$, and $I_{q,d}$ is $(\delta_1, \delta_2)$-regular with $\delta_1 = \frac{q^{d}-1}{q-1}$ and $\delta_2 = q^{d-1}$. Hence it suffices to show that $\lambda_{2}(I_{q,d}) = q^{\frac{d-1}{2}}$. To prove this, we will show 
    \begin{equation}\label{eq:incidence_eigenvalue}
        A^3 
        = (\delta_1-1)q^{d-2}
        \left[ 
        \begin{array}{c|c} 
          0 & J \\ 
          \hline 
          J^t & 0 
        \end{array} 
        \right]
        +
        \delta_2 A,
    \end{equation}
    where $A$ is the adjacency matrix of $I_{q,d}$ and $J$ is the $|\mathbb{F}_q^d| \times |\mathbb{H}_q^d|$ all-one matrix. Given this identity, it follows from a simple linear algebra argument that every eigenvalue of $A$ except for $\pm \sqrt{\delta_1\delta_2}$ satisfies the equation $x^3 = \delta_2 x$. In particular, $\lambda_{2}(I_{q,d}) = \sqrt{\delta_2} = q^{\frac{d-1}{2}}$ as claimed.
    
    To prove \eqref{eq:incidence_eigenvalue}, observe that the $(u,v)$-entry of $A^3$ equals the number of length-$3$ walks from $u$ to $v$. Therefore, $A^3(u,v) = 0$ if $u,v$ are both in $\mathbb{F}_q^d$ or $\mathbb{H}_q^d$. Since $A$ is a symmetric matrix, we can assume that $u\in \mathbb{F}_q^d$ and $v\in \mathbb{H}_q^d$. Let $uv'u'v$ be a length-$3$ walk from $u$ to $v$ in $I_{q,d}$. If $\{u,v\} \not \in E(I_{q,d})$, then there are $\delta_1-1$ choices for $v'$, and for each chosen $v'$ there are $q^{d-2}$ choices for $u'$ (since the intersection of two hyperplanes is either empty or a $(d-2)$-flat). If $\{u,v\}\in E(I_{q,d})$, then in addition to the above $(\delta_1-1)q^{d-2}$ choices, we can also choose $v' = v$ and in this case there are $\delta_2$ more choices for $u'$ (i.e. the neighbors of $v$). Hence we conclude the proof.
\end{proof}

Using Lemmas~\ref{expander_mixing} and~\ref{incidence_eigenvalue}, we obtain the following crucial property of the point-hyperplane incidence relation of $\mathbb{F}_{q}^{d}$.
\begin{lemma}\label{middle_lines}
    For any point set $P$ in $\mathbb{F}_{q}^{d}$, we have 
    \begin{equation*}
        \left|\left\{ \pi \in \mathbb{H}_{q}^d :~ |\pi \cap P| \not\in \left[\frac{|P|}{2q}, \frac{2|P|}{q}\right]  \right\}\right| \le \frac{8q^{d+1}}{|P|}.
    \end{equation*}
\end{lemma}
\begin{proof}
    Let $\mathbb{H}^{<} = \left\{ \pi \in \mathbb{H}_{q}^d :~ |\pi \cap P| < {|P|}/{(2q)}\right\}$ and $\mathbb{H}^{>} = \left\{ \pi \in \mathbb{H}_{q}^d :~ |\pi \cap P| > {2|P|}/{q}\right\}$. Applying Lemmas~\ref{expander_mixing} and~\ref{incidence_eigenvalue} to $I_{q,d}$ with $S = P$ and $T = \mathbb{H}^{<}$, we obtain \begin{equation*}
        e\left(P, \mathbb{H}^{<}\right) 
        \ge \frac{\delta_2}{|\mathbb{F}_{q}^d|}|P||\mathbb{H}^{<}| - \lambda_2(I_{q,d})\sqrt{|P||\mathbb{H}^{<}|} = \frac{|P||\mathbb{H}^{<}|}{q} - q^{\frac{d-1}{2}}\sqrt{|P||\mathbb{H}^{<}|}.
    \end{equation*}
    Notice that $e\left(P, \mathbb{H}^{<}\right) \leq |\mathbb{H}^{<}||P|/{(2q)}$. Combining these two inequalities we can solve \begin{equation*}
        \frac{|P||\mathbb{H}^{<}|}{2q} 
        \le q^{\frac{d-1}{2}}\sqrt{|P||\mathbb{H}^{<}|},
    \end{equation*} which implies that $|\mathbb{H}^{<}| \le {4q^{d+1}}/{|P|}$. Similarly, we can show that $|\mathbb{H}^{>}| \le {4q^{d+1}}/{|P|}$, hence we conclude the proof.
\end{proof}

Our balanced supersaturation result for collinear triples is a corollary of Lemma~\ref{middle_lines}. 
\begin{lemma}\label{supersaturation_Fq2}
    For any point set $P$ in $\mathbb{F}_q^2$, there is a collection $\mathcal{S}$ of collinear triples in $P$ such that
    \begin{enumerate}[(i)]
        \item $|\mathcal{S}| \ge \left(q(q+1) - 8q^3/|P|\right)\binom{|P|/(2q)}{3}$, 
        \item every point $v\in P$ is contained in at most $(q+1) \binom{2|P|/q}{2}$ elements in $\mathcal{S}$, and 
        \item every pair $\{u,v\}\subset P$ is contained in at most $2|P|/q$ elements in $\mathcal{S}$.
    \end{enumerate}
\end{lemma}
\begin{proof}
   Let
    \begin{equation*}
        L = \left\{ \ell\in \mathbb{H}_q^2 :~ |\ell\cap P| \in \left[\frac{|P|}{2q}, \frac{2|P|}{q}\right]  \right\} 
        \quad\text{and}\quad
        \mathcal{S} = \bigcup_{\ell\in L}\binom{\ell\cap P}{3}. 
    \end{equation*}
    According to Lemmas~\ref{flat_counting} and~\ref{middle_lines}, we have $|L|\ge q(q+1)-8q^3/|P|$. By the definition of $L$, every line $\ell \in L$ contains at least $|P|/(2q)$ points in $P$, so we obtain 
    \begin{equation*}
        |\mathcal{S}|\ge \left(q(q+1)-\frac{8q^3}{|P|}\right)\binom{|P|/(2q)}{3}, 
    \end{equation*} 
    and this concludes the proof of (i). 

    To prove (ii), notice that every point $v \in P$ is contained in at most $q+1$ lines in $L$, and moreover, there are at most $\binom{2|P|/q -1}{2}$ collinear triples in $\mathcal{S}$ containing $v$ in each line $\ell \in L$. So $v$ is contained in at most $(q+1) \binom{2|P|/q}{2}$ elements in $\mathcal{S}$.
    
    Finally, let $\ell_{uv}$ be the unique line in $\mathbb{F}_{q}^{2}$ determined by a pair $\{u,v\} \subset P$. If $\ell_{uv} \not\in L$, then no element in $\mathcal{S}$ contains $\{u,v\}$. If $\ell_{uv} \in L$, then by the definition of $L$, there are at most $2|P|/q -2$ elements in $\mathcal{S}$ containing $\{u,v\}$. This proves (iii).
\end{proof}

\subsection{The first container theorem}
In this subsection, we prove the following container theorem for general position sets in $\mathbb{F}^2_q$ using Lemma~\ref{supersaturation_Fq2}.

\begin{theorem}\label{thm:container1}
    There exists an absolute constant $c>0$ such that for sufficiently large prime power $q$, we have a collection $\mathcal{C}$ of point sets (called containers) in $\mathbb{F}^2_q$ satisfying
    \begin{enumerate}[(i)]
        \item every general position set of $\mathbb{F}^2_q$ is contained in some member of $\mathcal{C}$,
        \item $|C|\le 9q$ for all $C\in \mathcal{C}$, and 
        \item $|\mathcal{C}|\le \exp \l(c \cdot q^{1/2}(\log q)^2\r)$.
    \end{enumerate}
\end{theorem} 

Our proof will be an iterative application of Lemmas~\ref{hypergraphcontainer} and~\ref{supersaturation_Fq2}. We start with the container collection $\mathcal{C}$ consisting of the space $\mathbb{F}_q^2$ alone. We repeatedly apply Lemma~\ref{supersaturation_Fq2} with each large member $C \in \mathcal{C}$ to obtain a triple system $\mathcal{S}_C$, then apply Lemma~\ref{hypergraphcontainer} with $\mathcal{S}_C$ to obtain a new collection of containers $\mathcal{C}_C$, and replace $C$ in $\mathcal{C}$ with all members in $\mathcal{C}_C$. This iteration process ends when every member in $\mathcal{C}$ becomes small, and we can verify that the final collection satisfies the claimed properties.

\begin{proof}[Proof of Theorem~\ref{thm:container1}]
    We set $\mathcal{C}_{0} = \{\mathbb{F}_q^2\}$ and $\varepsilon = 1/e$. Suppose that $\mathcal{C}_j$ has been defined for some integer $j \ge 0$. Let
        \begin{equation*}
            \mathcal{C}_{j}^{>} = \left\{C\in \mathcal{C}_j :~ |C|>9q\right\}
            \quad\text{and}\quad
            \mathcal{C}_{j}^{\le} = \left\{C\in \mathcal{C}_j :~ |C|\le 9q\right\}. 
        \end{equation*}
    If $\mathcal{C}_{j}^{>} = \emptyset$, then we let $\mathcal{C} = \mathcal{C}_j$. Otherwise, for each $C\in \mathcal{C}_{j}^{>}$ we set
    \begin{equation*}
        k_C = {|C|}/{q} 
        \quad\text{and}\quad
        \tau_{C} = 10^8 \cdot k_C^{-1} q^{-1/2}. 
    \end{equation*} We can apply Lemma~\ref{supersaturation_Fq2} with $P=C$ and obtain a collection $\mathcal{S}_{C}$ of collinear triples in $C$ satisfying
    \begin{itemize}
        \item $|\mathcal{S}_{C}| \ge \left(q(q+1) - 8q^3/|C|\right)\binom{|C|/(2q)}{3}\ge 10^{-3} k_C^3 q^2 $, 
        \item $\Delta_1(\mathcal{S}_{C}) \le (q+1) \binom{2|C|/q}{2} \le 4 k_C^2 q$, and 
        \item $\Delta_2(\mathcal{S}_{C}) \le 2|P|/q = 2k_C$.
    \end{itemize} Here, the fact $|C| > 9q$ is used in the first property. As a consequence, we can compute
    \begin{align*}
        \Delta(\mathcal{S}_C,\tau_C)
            &= \frac{2^{\binom{3}{2}-1}|C|}{3|\mathcal{S}_C|}
            \left(\frac{\Delta_2(\mathcal{S}_C)}{\tau_C} + \frac{\Delta_3(\mathcal{S}_C)}{2\tau_C^2}\right) \\
            &\le \frac{4k_C q}{3\cdot 10^{-3} k_C^3 q^2}\left(\frac{2k_C}{10^8 k_C^{-1} q^{-1/2}}+\frac{1}{2\left(10^8k_C^{-1}q^{-1/2}\right)^2}\right)\\
            &= \frac{8}{3\cdot 10^{5} q^{1/2}} + \frac{4}{6\cdot 10^{13}}
                \le \frac{\varepsilon}{12\cdot 3!}. 
    \end{align*}
    Additionally, since $q$ is sufficiently large, we have $\tau_{C} \le \frac{1}{200\cdot 3 \cdot 3!}$. It follows from Lemma~\ref{hypergraphcontainer}, applied to $\mathcal{S}_C$, that there exists a collection $\mathcal{C}_{C}$ of subsets of $C$ such that 
    \begin{enumerate}[(a)]
        \item every independent set of $\mathcal{S}_C$ is contained in some member of $\mathcal{C}_C$, 
        \item the size of $\mathcal{C}_{C}$ satisfies 
                \begin{align*}
                    |\mathcal{C}_C|
                    &\le \exp\left( 1000\cdot 3 \cdot (3!)^3 \cdot |C| \cdot \tau_{C} \cdot \log(1/\varepsilon)\cdot \log(1/\tau_C)\right)\\
                    &\le \exp\left( 10^{18} \cdot k_{C} q \cdot k_{C}^{-1} q^{-1/2}\cdot \log e \cdot\log \left(k_C q^{1/2}/10^8\right)\right)\\ 
                    &\le \exp\left( 10^{18} \cdot q^{1/2} \log q\right),~\text{and}
                \end{align*}
        \item for every $C'\in \mathcal{C}_{C}$ the induced subgraph $\mathcal{S}_{C}[C']$ has size at most $\varepsilon|\mathcal{S}_{C}|$. 
    \end{enumerate} Having constructed $\mathcal{C}_C$ for each $C \in \mathcal{C}_{j}^{>}$, we define
    \begin{equation*}
        \mathcal{C}_{j+1} = \mathcal{C}_{j}^{\le}~\cup \bigcup_{C \in \mathcal{C}_{j}^{>}} \mathcal{C}_{C}, 
    \end{equation*} and repeat this process to $\mathcal{C}_{j+1}$. Recall that our iteration stops with $\mathcal{C} = \mathcal{C}_{j}$ when $\mathcal{C}_{j}^{>} = \emptyset$.

    We prove that $\mathcal{C}$ satisfies the claimed properties. First, using statement~(a) above, we can argue by induction on $j$ that every general position set of $\mathbb{F}_{q}^2$ is contained in some member of $\mathcal{C}_j$. Indeed, a general position set is either contained in some member of $\mathcal{C}_{j}^{\le}$ or is an independent set in $\mathcal{S}_{C}$ for some $C \in \mathcal{C}_{j}^{>}$. Also, it is obvious from our construction that $|C| \leq 9q$ for all $C \in \mathcal{C}$.

    It suffices for us to bound the size of the output $\mathcal{C}$. We claim that the iterative process above must stop after at most $x = \frac{\log(9/q)}{\log(1-10^{-4})}$ steps. Indeed, statement~(c) above implies that for every $C'\in \mathcal{C}_{C}$ we have 
        \begin{equation*}
            \sum_{v\in C \setminus C'}d_{\mathcal{S}_{C}}(v) \ge (1-\varepsilon)|\mathcal{S}_{C}|.
        \end{equation*}
    Combining this with the properties of $\mathcal{S}_C$ concluded above, we have 
    \begin{equation*}
        |C\setminus C'|
        \ge \frac{(1-\varepsilon)|\mathcal{S}_{C}|}{\Delta_1(\mathcal{S}_{C})}
        \ge \frac{k_C q}{10^4}
        = \frac{|C|}{10^4}. 
    \end{equation*}
    Hence we have $|C'| \le \left(1-10^{-4}\right)|C|$. This implies that every container would have size at most $\left(1-10^{-4}\right)^{x}q^2 \le 9q$ after $x$ steps, so the process should not continue beyond. Combining this fact with statement~(b) above, we obtain
    \begin{equation*}
        |\mathcal{C}|
        \le \left(\exp\left(10^{18} \cdot q^{1/2} \log q\right)\right)^x
        \le \exp\left(\frac{10^{18}}{\log((1-10^{-4})^{-1})} \cdot q^{1/2} (\log q)^2 \right), 
    \end{equation*} and this concludes our proof.
\end{proof}

\subsection{Proof of Theorem~\ref{thm:randomturan1}}

We finish the proof of Theorem~\ref{thm:randomturan1} in this subsection. First, we prove the following upper bounds for $\alpha(\mathbb{F}_q^2,p)$.

\begin{proposition}\label{randomturan1_upper}
There exists an absolute constant $c>0$ such that a.a.s.
\begin{equation*}
\alpha(\mathbb{F}^2_q,p)\le
\l\{
\begin{aligned}
& cq^{1/2}(\log q)^2,~~&0\leq p\leq q^{-1/2}(\log q)^2,\\
& cpq,~~&q^{-1/2}(\log q)^2\le p\le 1.
\end{aligned}
\r.
\end{equation*}Furthermore, we have $\alpha(\mathbb{F}^2_q,p)\le (1+o(1))pq^2$ a.a.s. for $p \geq q^{-2}\log q$.
\end{proposition}
\begin{proof}
Let $\mathbf{S}_p$ be a $p$-random set of $\mathbb{F}_{q}^2$ and $\mathbf{X}_m(p)$ denote the number of general position sets of size $m$ in $\mathbf{S}_p$. It follows from Chernoff's bound that $|\mathbf{S}_p| \leq (1+o(1))pq^2$ a.a.s. for $p \geq q^{-2}\log q$. The ``Furthermore'' part of our claim is due to the fact $\alpha(\mathbb{F}^2_q,p) \le |\mathbf{S}_p|$. So it suffices to prove the first two inequalities.

By Theorem~\ref{thm:container1}, there exists a collection $\mathcal{C}$ of point sets in $\mathbb{F}^2_q$ satisfying
\begin{enumerate}[(i)]
        \item every general position set of $\mathbb{F}^2_q$ is contained in some member of $\mathcal{C}$, 
        \item $|C|\le 9q$ for all $C\in \mathcal{C}$, and 
        \item $|\mathcal{C}|\le \exp \l(cq^{1/2}(\log q)^2\r)$.
\end{enumerate} Here, $c$ is an absolute constant and without loss of generality, we may assume $c \ge 36e$.

First we consider the range $p\ge q^{-1/2}(\log q)^2$. We set $m = cpq$. By the linearity of expectation and statements (i)-(iii) above, we can estimate 
    \begin{equation}\label{eq:randomturan1_upper}
        \begin{aligned}
            \mathbb{E}[\mathbf{X}_m(p)] \le \sum_{C\in \mathcal{C}}\binom{|C|}{m} p^m
            \le |\mathcal{C}|\binom{9q}{m}p^m
            &\le \exp\left( cq^{1/2}(\log q)^2 \right) \left(\frac{9eq}{m}\right)^m p^m\\
            &\le \exp\l( cq^{1/2}(\log q)^2+m \cdot \log\l(\frac{9epq}{m}\r) \r)\\
            &\le \exp\l( m+m \cdot \log\l(\frac{1}{4}\r)\r) \\
            &= \exp(-cpq) \rightarrow 0 \quad \text{as}\quad q\rightarrow\infty.
        \end{aligned}
    \end{equation}
So it follows from Markov's inequality that
\begin{equation*}
    \Pr\left[\alpha(\mathbb{F}^2_q,p)\ge m\right] = \Pr\left[\mathbf{X}_m(p) \ge 1\right] \le \mathbb{E}[\textbf{X}_m(p)]  \rightarrow 0 \quad \text{as}\quad q\rightarrow\infty.
\end{equation*} Therefore, a.a.s. we have $\alpha(\mathbb{F}^2_q,p) < m = cpq$.

Now we consider the range $p \leq q^{-1/2}(\log q)^2$. Set $m = cq^{1/2}(\log q)^2$. Since $\mathbb{E}[\mathbf{X}_m(p)]$ is nondecreasing in $p$, it follows from \eqref{eq:randomturan1_upper} that 
\begin{equation*}
    \mathbb{E}\left[\mathbf{X}_m(p)\right] \le \mathbb{E}\left[\mathbf{X}_m\left(q^{-1/2}(\log q)^2\right)\right] \le \exp\left(-cq^{1/2}(\log q)^2\right) \rightarrow 0 \quad \text{as}\quad q\rightarrow\infty.
\end{equation*}
Similarly, it follows from Markov's inequality that a.a.s. $\alpha(\mathbb{F}^2_q,p) < m=cq^{1/2}(\log q)^2$.
\end{proof}

To finish the proof of Theorem~\ref{thm:randomturan1}, it suffices to prove the lower bounds for $\alpha(\mathbb{F}_q^2,p)$.

\begin{proof}[Proof of Theorem~\ref{thm:randomturan1}]
    Let $\mathbf{S}_p$ be a $p$-random set of $\mathbb{F}_{q}^2$ and $\mathbf{T}(p)$ denote the number of collinear triples in $\mathbf{S}_p$.
    
    We first consider the range $q^{-2}\log q\le p\le q^{-3/2}/\log q$. Since the total number of collinear triples in $\mathbb{F}_q^2$ is $(1+o(1))q^5/6$ (see Lemma~\ref{flat_counting}), we have $\mathbb{E}[\mathbf{T}(p)] = (1+o(1))p^3q^5/6$. So it follows from Markov's inequality that  
    \begin{equation*}
           \Pr\left[\mathbf{T}(p) \ge \frac{pq^2}{2}\right] \le (1+o(1))p^2q^3/3 \rightarrow 0 \quad \text{as}\quad q\rightarrow\infty.
    \end{equation*} On the other hand, it follows from Chernoff's bound that a.a.s. $|\mathbf{S}_p| = (1+o(1))pq^2$.
    Therefore, by deleting a point from each collinear triple in $\mathbf{S}_p$, we obtain a general position subset of $\mathbf{S}_p$ of size at least $(\frac{1}{2} + o(1))pq^2$. Hence, we have a.a.s.
    \begin{equation}\label{eq:randomturan1}
        \alpha(\mathbb{F}^2_q,p)
        \ge \left(\frac{1}{2} + o(1)\right)pq^2.
    \end{equation}

    For the range $q^{-3/2}/\log q \le p \leq q^{-1/2}(\log q)^2$, it follows from the nondecreasingness of $\alpha(\mathbb{F}^2_q,p)$ and~\eqref{eq:randomturan1} that a.a.s.
    \begin{equation*}
        \alpha(\mathbb{F}^2_q,p)
        \ge \alpha\left(\mathbb{F}^2_q,\frac{q^{-3/2}}{\log q}\right)
         \ge \left(\frac{1}{2}+o(1)\right)\frac{q^{1/2}}{\log q}.
    \end{equation*}
    
    Finally we consider the range $q^{-1/2}(\log q)^2\le p\le 1$. Let $\mathbf{Y} = \mathbf{S}_p \cap \left\{(x,x^2) :~ x\in \mathbb{F}_{q}\right\}$. By Chernoff's bound, we have a.a.s. $|\mathbf{Y}|=(1+o(1))pq$. As a consequence, $\alpha(\mathbb{F}^2_q,p) \ge |\textbf{Y}| = (1+o(1))pq$ a.a.s. in this range.
    
    These lower bounds together with Proposition~\ref{randomturan1_upper} complete the proof of Theorem~\ref{thm:randomturan1}.
\end{proof}

\section{Counting general position sets in \texorpdfstring{$\mathbb{F}_q^d$}{}}\label{sec:counting}

We aim to prove Theorems~\ref{thm:counting1},~\ref{thm:counting2}, and~\ref{thm:counting3} in this section.

\subsection{Supersaturation for coplanar tuples in \texorpdfstring{$\mathbb{F}_q^d$}{}}

The main result of this subsection is the following supersaturation result for coplanar $(d+1)$-tuples in large point sets of $\mathbb{F}_{q}^d$. Recall that a \textit{coplanar $(d+1)$-tuple} consists of $d+1$ points contained in some hyperplane.

\begin{lemma}\label{supersaturation_Fqd}
    For every integer $d\ge2$ there exists a constant $c>0$ such that the following holds for all prime power $q \ge 2(d+1)$. Suppose $P$ is a point set of size at least $q+2(d+1)$ in $\mathbb{F}_{q}^d$, then the number of coplanar $(d+1)$-tuples contained in $P$ is at least \begin{equation*}
        c \cdot \min\left\{\frac{|P|}{q}-1,1\right\} \cdot \frac{|P|^{d+1}}{q}.
    \end{equation*}
\end{lemma}
\begin{proof}
    Let $P$ be as given in our hypothesis. For each $(d-1)$-tuple $A\subset P$, we fix a $(d-2)$-flat $F_{A}$ in $\mathbb{F}_q^d$ that contains $A$. Notice that each $F_A$ is contained in exactly $q+1$ hyperplanes, denoted by $\pi_A^1,\pi_A^2\dots,\pi_A^{q+1}$, and these hyperplanes cover the space $\mathbb{F}_q^d$. Hence we have\begin{equation*}
        |F_{A}\cap (P\setminus A)| + \sum_{i=1}^{q+1}|(\pi_{A}^i\setminus F_{A})\cap P| = |P\setminus A|.
    \end{equation*}
    Observe that every pair of points in $\pi_A^i \setminus F_A$ or $F_A\setminus A$ together with $A$ forms a coplanar $(d+1)$-tuple. By Jensen's inequality, the number of coplanar $(d+1)$-tuples in $P$ containing $A$ is at least
    \begin{equation*}
        \binom{|F_{A}\cap (P\setminus A)|}{2} + \sum_{i=1}^{q+1}\binom{|(\pi_{A}^i\setminus F_{A})\cap P|}{2}
         \ge (q+2) \binom{\frac{|P \setminus A|}{q+2}}{2}
         = \frac{\left(|P|-d+1\right)\left(|P|-d-1-q\right)}{2(q+2)}.    
        \end{equation*}
    Note that $|P|\ge q+2(d+1)$ and $q\ge 2(d+1)$ by our hypothesis. If $|P|\le 2q$, we can check
    \begin{equation*}
        |P|-d-1-q \geq \frac{|P|}{2}-\frac{q}{2} = \l(\frac{|P|}{q}-1\r)\frac{q}{2} \ge \left(\frac{|P|}{q}-1\right)\frac{|P|}{4}. 
    \end{equation*} If $|P|\ge 2q$, then we can check $|P|-d-1-q \ge {|P|}/{4}$. In both cases, we have\begin{equation*}
        |P|-d-1-q\ge \min\left\{\frac{|P|}{q}-1,1\right\} \cdot \frac{|P|}{4}.
    \end{equation*} Combining above inequalities, we conclude that the number of coplanar $(d+1)$-tuples in $P$ containing $A$ is at least $\min\left\{{|P|}/{q}-1,1\right\} \cdot |P|^2/(32q)$.

    Averaging over all $(d-1)$-tuples in $P$, the number of coplanar $(d+1)$-tuples contained in $P$ is at least
    \begin{equation*}
        \frac{\binom{|P|}{d-1}}{\binom{d+1}{d-1}} \cdot \min\left\{\frac{|P|}{q}-1,1\right\} \cdot \frac{|P|^2}{32q}   
        \ge c \cdot \min\left\{\frac{|P|}{q}-1,1\right\} \cdot \frac{|P|^{d+1}}{q}, 
    \end{equation*}
    where $c>0$ is a sufficiently small constant depending only on $d$.
\end{proof}

In order to apply the supersaturation result succinctly in upcoming computations, we introduce some further notations and prove a refined version of Lemma~\ref{supersaturation_Fqd}. We use $\mathcal{H}_{q,d}$ to denote the $(d+1)$-graph whose vertices are all the points of $\mathbb{F}_{q}^d$ and edges are all the coplanar $(d+1)$-tuples. Given an $r$-graph $\mathcal{H}$ and two real numbers $c, \tau>0$, we say $\mathcal{H}$ is \textit{$(c,\tau)$-bounded} if
    \begin{equation*}
        \Delta_i(\mathcal{H})
        \le \frac{c|\mathcal{H}|}{|V(\mathcal{H})|}\tau^{i-1}
        \quad\text{for all}\quad 2 \le i \le r.
    \end{equation*}

\begin{lemma}\label{refined_supersaturation_Fqd}
    For every integer $d\ge 2$ there exists a constant $c>0$ such that the following holds for all prime power $q \ge 2(d+1)$. Suppose $P$ is a point set of size at least $q+2(d+1)$ in $\mathbb{F}_{q}^d$, then the induced subgraph of $\mathcal{H}_{q,d}$ on $P$ is $(c, \tau)$-bounded, where 
    \begin{equation*}
        \tau = \left(\min\{k-1,1\}\right)^{-\frac{d}{d+1}} k^{-1} q^{-\frac{1}{d+1}}
        \quad\text{and}\quad
        k = {|P|}/{q}. 
    \end{equation*}
\end{lemma}
\begin{proof}
    Let $P$ be as given in our hypothesis, $\mathcal{H} = \mathcal{H}_{q,d}[P]$ be the induced subgraph, and $c_1>0$ be the constant guaranteed by Lemma~\ref{supersaturation_Fqd}. For every $j \in [d-1]$, we claim that the intersection of $P$ and a $j$-flat $F_j$ has size at most $(d+1)|\mathcal{H}|^{1/(d+1)}$. Otherwise, since every $(d+1)$-tuple in $F_j$ is coplanar, we would have \begin{equation*}
        |\mathcal{H}| \ge \binom{(d+1)|\mathcal{H}|^{1/(d+1)}}{d+1} > |\mathcal{H}|,
    \end{equation*} which is a contradiction.
    
    We consider the upper bound for $\Delta_2(\mathcal{H})$ first. Fix a pair of points $v_1,v_2\in P$, for each sequence $(v_1,v_2, \dots, v_{d+1})$ that forms a coplanar $(d+1)$-tuple in $P$, we define $\psi(v_1,v_2,\dots,v_{d+1})$ to be the smallest integer $i$ such that $v_1,\dots,v_i$ are affinely dependent. It follows from the claim above that the preimage $\psi^{-1}(i)$ has size at most $(d+1)|\mathcal{H}|^{{1}/{(d+1)}}|P|^{d-2}$. Combined with Lemma~\ref{supersaturation_Fqd}, we see that the number of coplanar $(d+1)$-tuples containing $\{v_1, v_2\}$ is at most 
    \begin{align*}
        \sum_{i=3}^{d+1} |\psi^{-1}(i)| &\leq (d-1) \cdot (d+1)|\mathcal{H}|^{\frac{1}{d+1}}|P|^{d-2}\\
        & = (d^2-1) |\mathcal{H}|^{\frac{1}{d+1}} 
            \cdot 
            \frac{\left(c_1\cdot \min\left\{k-1, 1\right\}\right)^{\frac{d}{d+1}} \cdot |P|^{d}/q^{\frac{d}{d+1}}}{\left(c_1\cdot \min\left\{k-1, 1\right\}\right)^{\frac{d}{d+1}} \cdot |P|^{2}/q^{\frac{d}{d+1}}} \\
        & \le (d^2-1) |\mathcal{H}|^{\frac{1}{d+1}} 
            \cdot \frac{|\mathcal{H}|^{\frac{d}{d+1}}}{\left(c_1\cdot \min\left\{k-1, 1\right\}\right)^{\frac{d}{d+1}} \cdot |P|\cdot(kq)/q^{\frac{d}{d+1}}} \\
        & = (d^2-1) {c_1^{-\frac{d}{d+1}}} \cdot \frac{|\mathcal{H}|}{|P|}\tau.
    \end{align*}
    Now we consider $\Delta_j(\mathcal{H})$ for $j \geq 3$. Using the fact $\Delta_j(\mathcal{H}) \leq |P|^{d+1-j}$ and Lemma~\ref{supersaturation_Fqd}, we have
    \begin{align*}
         \Delta_j(\mathcal{H})
        & \le \frac{c_1\cdot \min\left\{k-1, 1\right\} \cdot |P|^{d+1}/q}{c_1\cdot \min\left\{k-1, 1\right\} \cdot |P|^j /q} \\
        & \le \frac{|\mathcal{H}|}{c_1\cdot \min\left\{k-1, 1\right\} \cdot |P| \cdot (kq)^{j-1}/q} \\
        & \le \frac{1}{c_1 \cdot \min\{k-1,1\}^{1-\frac{(j-1)d}{d+1}} q^{j-2 - \frac{j-1}{d+1}}} \cdot \frac{|\mathcal{H}|}{|P|} \tau^{j-1} \\
       & \le  \frac{1}{c_1}\cdot \frac{|\mathcal{H}|}{|P|} \tau^{j-1},
    \end{align*}
    where the last inequality follows from $1-\frac{(j-1)d}{d+1} \leq 0$ and $j-2 - \frac{j-1}{d+1} \geq 0$. Hence we conclude the proof by taking $c$ to be a sufficiently large constant depending on $d$ and $c_1$.
\end{proof}

\subsection{The second container theorem}
We prove the following container theorem for general position sets in $\mathbb{F}^d_q$ using Lemmas~\ref{supersaturation_Fqd} and~\ref{refined_supersaturation_Fqd}. 

\begin{theorem}\label{thm:container2}
    For every integer $d\ge 2$ there exists a constant $c>0$ such that for sufficiently large prime power $q$ and real number $k \geq 1+cq^{-1/d}$, we have a collection $\mathcal{C}=\mathcal{C}(q,k)$ of point sets (called containers) in $\mathbb{F}^d_q$ satisfying
    \begin{enumerate}[(i)]
        \item every general position set in $\mathbb{F}^d_q$ is contained in some member of $\mathcal{C}$,
        \item $|C|\le kq$ for all $C\in \mathcal{C}$, and
        \item $|\mathcal{C}|\le \exp \l( c \cdot \l(q/\min\left\{k-1,~1\right\} \r)^{\frac{d}{d+1}} (\log q)^2\r)$.
    \end{enumerate}
\end{theorem}

It is worth highlighting that in Theorem~\ref{thm:container2}~(ii),  we have the flexibility to minimize the container size to as little as $q+cq^{\frac{d-1}{d}} = (1+o(1))q$, while in Theorem~\ref{thm:container1}~(ii) we can only guarantee an upper bound $9q$ for the container size. However, the upper bound for the total number of containers guaranteed by Theorem~\ref{thm:container2}~(iii) is worse than that in Theorem~\ref{thm:container1}~(iii). 

\begin{proof}
    Let $k > 1+cq^{-1/d}$ be fixed with $c$ being a sufficiently large constant (depending on $d$). We set $\mathcal{C}_{0} = \{\mathbb{F}_q^d\}$ and $\varepsilon = 1/e$. Suppose that $\mathcal{C}_j$ has been defined for some integer $j \ge 0$. Let \begin{equation*}
        \mathcal{C}_{j}^{>} = \left\{C\in \mathcal{C}_j :~ |C|>kq\right\}
            \quad\text{and}\quad
        \mathcal{C}_{j}^{\le} = \left\{C\in \mathcal{C}_j :~ |C|\le kq\right\}. 
    \end{equation*}
    If $\mathcal{C}_{j}^{>} = \emptyset$, then we let $\mathcal{C} = \mathcal{C}_j$. Otherwise, for each $C\in \mathcal{C}_{j}^{>}$ we set
    \begin{equation*}
        k_C = {|C|}/{q},\quad 
        \tau_{C} = \left(\min\{k_{C}-1,1\}\right)^{-\frac{d}{d+1}}k_{C}^{-1} q^{-\frac{1}{d+1}}, 
        \quad\text{and}\quad
        \hat{\tau}_C = c_1\cdot 2^{\binom{d+1}{2}}\cdot 12(d+1)!\cdot \tau_C, 
    \end{equation*}
    where $c_1>0$ is the constant guaranteed by Lemma~\ref{refined_supersaturation_Fqd}. Furthermore, let $\mathcal{H}_C = \mathcal{H}_{q,d}[C]$ be the induced subgraph of $\mathcal{H}_{q,d}$ on $C$, and recall that $\mathcal{H}_{q,d}$ is the hypergraph defined by coplanar $(d+1)$-tuples in $\mathbb{F}_{q}^d$. Since $|C| > kq \ge q+2(d+1)$, it follows from Lemma~\ref{refined_supersaturation_Fqd} that $\mathcal{H}_C$ is $(c_1, \tau_C)$-bounded. By the definitions of $(c_1, \tau_C)$-boundedness and $\Delta\left(\mathcal{H}_C,\hat{\tau}_C\right)$, we can compute
     \begin{align*}
            \Delta\left(\mathcal{H}_C,\hat{\tau}_C\right)
            = \frac{2^{\binom{d+1}{2}-1}|C|}{(d+1)|\mathcal{H}_C|} \sum_{i=2}^{d+1}\frac{\Delta_{i}\left(\mathcal{H}_C\right)}{2^{\binom{i-1}{2}} \hat{\tau}_C^{i-1}}
            &\le \frac{2^{\binom{d+1}{2}-1}|C|}{(d+1)|\mathcal{H}_C|} \sum_{i=2}^{d+1} \frac{|\mathcal{H}_C|}{|C|} \frac{c_1 \tau_C^{i-1}}{2^{\binom{i-1}{2}} \hat{\tau}_C^{i-1}}\\
             &\le \frac{2^{\binom{d+1}{2}-1}}{d+1} \sum_{i=2}^{d+1} \frac{1}{2^{\binom{i-1}{2}} \cdot 2^{\binom{d+1}{2}}\cdot 12(d+1)!}\\
            &\le  \frac{\varepsilon}{12(d+1)!}.
    \end{align*}
    In addition, since $c$ and $q$ are sufficiently large, we can check that $\hat{\tau}_C \le \frac{1}{200\cdot (d+1) \cdot (d+1)!}$. It follows from Lemma~\ref{hypergraphcontainer}, applied to $\mathcal{H}_C$, that there exists a collection $\mathcal{C}_{C}$ of subsets of $C$ such that
    \begin{enumerate}[(a)]
        \item every independent set of $\mathcal{H}_C$ is contained in some member of $\mathcal{C}_C$,
        \item the size of $\mathcal{C}_C$ satisfies 
            \begin{align*}
                |\mathcal{C}_C| & \leq \exp\l(1000\cdot (d+1) \cdot \left((d+1)!\right)^3 \cdot |C| \cdot \hat{\tau}_C\cdot \log\left(\frac{1}{\varepsilon}\right)\cdot \log\left(\frac{1}{\hat{\tau}_C}\right)\r) \\
                & \le \exp\l(c_2 \cdot \l(q/\min\left\{k-1,~1\right\} \r)^{\frac{d}{d+1}} \log q \r),
            \end{align*} 
            where $c_2$ is a large constant depending on $d$ and the inequality follows from the definitions of $\hat{\tau}_C$ and $k_C$, and
        \item for every $C' \in \mathcal{C}_C$ the induced subgraph of $\mathcal{H}_C$ on $C'$ has size at most $\varepsilon|\mathcal{H}_C|$.
    \end{enumerate} Having constructed $\mathcal{C}_C$ for each $C \in \mathcal{C}_{j}^{>}$, we define
    \begin{equation*}
        \mathcal{C}_{j+1} = \mathcal{C}_{j}^{\le}~\cup \bigcup_{C \in \mathcal{C}_{j}^{>}} \mathcal{C}_{C}, 
    \end{equation*} and repeat this process to $\mathcal{C}_{j+1}$. Recall that our iteration stops with $\mathcal{C} = \mathcal{C}_{j}$ when $\mathcal{C}_{j}^{>} = \emptyset$.

    Similar to the proof of Theorem~\ref{thm:container1}, we can verify the statements (i) and (ii) claimed for our container collection $\mathcal{C}$. So it suffices for us to bound the size of the output $\mathcal{C}$. We claim that the iterative process above must stop after at most $x = d(d+1) \cdot \log q$ steps. Indeed, statement~(c) above implies that the induced subgraph of $\mathcal{H}_{q,d}$ on any $C\in \mathcal{C}_{x}^{>}$ has size at most $\varepsilon^{x} |\mathcal{H}_{q,d}|< e^{-x} \left(q^d\right)^{d+1} = 1$. On the other hand, Lemma~\ref{supersaturation_Fqd} implies that this induced subgraph contains at least one edge when $q$ is sufficiently large. This contradiction means the process should not continue beyond $x$ steps. Combining this fact with statement~(b) above, we obtain
    \begin{equation*}
        |\mathcal{C}|
        \leq \exp\l(c_2 \cdot \l(q/\min\left\{k-1,~1\right\} \r)^{\frac{d}{d+1}} \log q \cdot x\r)
        \le \exp\l(c \cdot \l(q/\min\left\{k-1,~1\right\} \r)^{\frac{d}{d+1}} \left(\log q\right)^2\r),
    \end{equation*} where the last inequality follows from $c$ being sufficiently large.
\end{proof}

\subsection{Proofs of Theorems~\ref{thm:counting1},~\ref{thm:counting2}, and~\ref{thm:counting3}}

We finish the proofs of our counting results using Theorems~\ref{thm:container1} and~\ref{thm:container2}. 
 
\begin{proof}[Proof of Theorem~\ref{thm:counting1}]
Let $c$ be the constant guaranteed by Theorem~\ref{thm:container2}. Fix $k = 1+q^{-\frac{1}{2d+1}}$ and since $q$ is sufficiently large, we have $k>1+cq^{-\frac{1}{d}}$. Let $\mathcal{C} = \mathcal{C}(q,k)$ be the container collection outputted by Theorem~\ref{thm:container2}. Note that $\min\left\{k-1,~1\right\} = q^{-\frac{1}{2d+1}}$. By the properties of $\mathcal{C}$ guaranteed by Theorem~\ref{thm:container2}, the number of general position sets in $\mathbb{F}_q^d$ is at most \begin{equation*}
    \sum_{C\in \mathcal{C}} 2^{|C|} 
    \le |\mathcal{C}| 2^{kq}
    \le \exp\l( c\l(q^{\frac{1}{2d+1}} \cdot q\r)^{\frac{d}{d+1}}(\log q)^2 + \l(1+q^{-\frac{1}{2d+1}}q \r) \r)
    = 2^{q + q^{\frac{2d}{2d+1}+o(1)}}. \qedhere
\end{equation*}
\end{proof}

\begin{proof}[Proof of Theorem~\ref{thm:counting2}]

Let $m$ be as given in our hypothesis and we assume $m \le 2q$ (otherwise the desired bound is trivial). Define \begin{equation*}
        \epsilon = m^{-\frac{d+1}{2d+1}}q^{\frac{d}{2d+1}}(\log q)^{\frac{2d+2}{2d+1}}.
    \end{equation*}
Using this definition and the bounds on $m$, we can check that \begin{equation*}
    \left(\frac{q}{\epsilon}\right)^{\frac{d}{d+1}} (\log q)^2 = \epsilon m
    \quad \text{and} \quad
    q^{-\frac{1}{2d+1}}\left(\frac{\log q}{2}\right)^{\frac{d+1}{2d+1}}\le \epsilon \le \left(\log q\right)^{-\frac{d+1}{2d+1}}.
\end{equation*} Hence we set $k = 1+ \epsilon$ and notice that $k > 1+cq^{-1/d}$ and $\min\left\{k-1,~1\right\} = \epsilon$ for sufficiently large $q$. Let $\mathcal{C} = \mathcal{C}(q,k)$ be the container collection outputted by Theorem~\ref{thm:container2}. By Theorem~\ref{thm:container2}~(iii) and the identity above, we have
\begin{equation*}
    |\mathcal{C}| 
    \le \exp\l(c \left(\frac{q}{\epsilon}\right)^{\frac{d}{d+1}}(\log q)^2\r)
    = \exp\l(c \epsilon m\r),
\end{equation*} where $c$ is the constant guaranteed by Theorem~\ref{thm:container2}. By the other properties of $\mathcal{C}$, the number of general position sets of size $m$ in $\mathbb{F}_q^d$ is at most \begin{equation*}
    \sum_{C\in \mathcal{C}} \binom{|C|}{m}
    \le |\mathcal{C}| \binom{kq}{m}
    \le \exp\l(c \epsilon m\r) \binom{(1+\epsilon)q}{m}
    \le \binom{2^{c \epsilon}(1+\epsilon)q}{m}.
\end{equation*}
Here we used the the inequality $\binom{2^{c \epsilon}(1+\epsilon)q}{m}/\binom{(1+\epsilon)q}{m} \ge \left(\frac{2^{c \epsilon}(1+\epsilon)q}{(1+\epsilon)q}\right)^m =\exp\l(c \epsilon m\r)$. By the bound $\epsilon \le \left(\log q\right)^{-\frac{d+1}{2d+1}}$, we can argue $2^{c \epsilon}(1+\epsilon)q = (1+o(1))q$ and this concludes our proof.
\end{proof}

\begin{proof}[Proof of Theorem~\ref{thm:counting3}]
Let $m$ be as given in our hypothesis and $\mathcal{C}$ be the container collection outputted by Theorem~\ref{thm:container1}. Similar to the proof of Theorem~\ref{thm:counting2}, the number of general position sets of size $m$ in $\mathbb{F}_q^2$ is at most \begin{equation*}
    \sum_{C\in \mathcal{C}} \binom{|C|}{m}
    \le |\mathcal{C}| \binom{9q}{m}
    \le \exp\l(cq^{1/2}(\log q)^2\r) \binom{9q}{m}
    \le \binom{2^c \cdot 9q}{m},
\end{equation*} where $c$ is the constant guaranteed by Theorem~\ref{thm:container1}. Here we used the inequality that $\binom{2^c \cdot 9q}{m} / \binom{9q}{m} \ge \left(\frac{2^c \cdot 9q}{9 q}\right)^m \ge \exp\left(cq^{1/2}(\log q)^2\right)$. 
\end{proof}

\section{Random Tur{\'a}n results in \texorpdfstring{$\mathbb{F}_q^d$}{}}\label{sec:randomturan2}

We aim to prove Theorem~\ref{thm:randomturan2} in this section.

\subsection{Supersaturation for critical coplanar tuples in \texorpdfstring{$\mathbb{F}_q^d$}{}}
We say a coplanar $(d+1)$-tuple in $\mathbb{F}_{q}^d$ is \textit{critical} if every proper subset of it is affinely independent. The main result of this subsection is the following supersaturation result for critical coplanar $(d+1)$-tuples. 

\begin{lemma}\label{critical_supersaturation_Fqd}
    For every integer $d\ge2$ there exists a constant $c>0$ such that the following holds for all prime power $q$. Suppose $P$ is a point set of size at least $4d q^{d-1}$ in $\mathbb{F}_{q}^d$, then $P$ contains at least $c|P|^{d+1}/q$ critical coplanar $(d+1)$-tuples.
\end{lemma}
\begin{proof}
    Let $P$ be as given in our hypothesis, $N = q^{d-k} \binom{d}{k}_q = (1+o(1))q^d$ (see Lemma~\ref{flat_counting}), and label all hyperplanes in $\mathbb{F}_q^d$ as $\pi_1,\pi_2,\dots,\pi_N$. A simple double-counting and Lemma~\ref{flat_counting} give 
    \begin{equation*}
        \sum_{i=1}^N |P \cap \pi_i| 
        = \sum_{v \in P} \text{ number of hyperplanes containing $v$ }
        = \binom{q}{d-1}_q|P|
        \geq q^{d-1}|P|.
    \end{equation*}
    
    Define the function
    \begin{equation*}
        f(x):=
        \begin{cases}
            \frac{1}{(d+1)!} x \cdot \prod_{i=0}^{d-2}\left(x - q^{i}\right) \cdot \left(x - d q^{d-2}\right), & \quad\text{if}\quad x \ge 2 d q^{d-2}, \\
            0, & \quad\text{if}\quad x < 2 d q^{d-2}. 
        \end{cases}
    \end{equation*} By choosing points one by one, it is easy to see that if a hyperplane $\pi$ contains $x$ points from $P$, then there are at least $f(x)$ critical coplanar $(d+1)$-tuples in $P\cap \pi$. Therefore, it follows from the convexity of $f(x)$ and Jensen's inequality that the number of critical coplanar $(d+1)$-tuples in $P$ is at least
    \begin{align*}
        \sum_{i=1}^N f\left(|P \cap \pi_i|\right) 
        &\geq N \cdot f\l(\frac{\sum_{i=1}^N|P \cap \pi_i|}{N}\r)
          \geq N \cdot f\l(\frac{q^{d-1}|P|}{N}\r) 
        \geq q^d \cdot f\l(\frac{|P|}{2q}\r) \\
        & \ge \frac{q^d}{(d+1)!}\left(\frac{|P|}{4q}\right)^{d+1} = \frac{1}{4^{d+1}(d+1)!} \cdot \frac{|P|^{d+1}}{q}. \qedhere
    \end{align*}
\end{proof}

We remark that the number of critical coplanar $(d+1)$-tuples guaranteed by Lemma~\ref{critical_supersaturation_Fqd} is tight up to some constant factor. To see this, consider a $p$-random set $\textbf{S}_p$ of $\mathbb{F}_{q}^d$. By Chernoff's bound, a.a.s. $|\textbf{S}_p| = (1+o(1))pq^{d}$ as long as $p \gg 1/q^d$. Since the number of coplanar $(d+1)$-tuples inside $\mathbb{F}_q^d$ is $(1+o(1))q^{d+(d-1)(d+1)}$ (because there are $q^d$ hyperplanes in $\mathbb{F}_q^d$ and
each one contains $(1+o(1))q^{(d-1)(d+1)}$ coplanar $(d+1)$-tuples), the number of coplanar $(d+1)$-tuples in $\textbf{S}_p$ is a.a.s. $(1+o(1))q^{d+(d-1)(d+1)}p^{d+1} = O\left({|\textbf{S}_p|^{d+1}}/{q}\right)$.

We also prove a refined version of Lemma~\ref{critical_supersaturation_Fqd} which is more handy in later computations. We use $\mathcal{H}'_{q,d}$ to denote the $(d+1)$-graph whose vertices are all the points of $\mathbb{F}_{q}^d$ and edges are all the critical coplanar $(d+1)$-tuples.

\begin{lemma}\label{refined_critical_supersaturation_Fqd}
    For every integer $d\ge 2$ there exists a constant $c>0$ such that the following holds for all prime power $q$. Suppose $P$ is a point set of size at least $4dq^{d-1}$ in $\mathbb{F}_{q}^d$, then the induced subgraph of $\mathcal{H}'_{q,d}$ on $P$ is $(c, \tau)$-bounded, where
    \begin{equation*}
        \tau = \max\l\{ k^{-1} q^{\frac{1}{d}-1},~ k^{-2}q^{d-2}\r\}
        \quad\text{and}\quad 
        k= {|P|}/{q}.
    \end{equation*}    
\end{lemma}
\begin{proof}
    Let $P$ be as given in our hypothesis, $\mathcal{H} = \mathcal{H}'_{q,d}[P]$ be the induced subgraph, and $c_1>0$ be the constant guaranteed by Lemma~\ref{critical_supersaturation_Fqd}.
    
    Let $2\leq j\leq d$ be an integer and $v_1, \ldots, v_j \in P$ be affinely independent. To upper bound $\Delta_j(\mathcal{H})$ it suffices to count the number of choices for points $v_{j+1}, \ldots, v_{d+1} \in P$ such that $\{v_1, \ldots, v_{d+1}\}$ is a critical $(d+1)$-tuple. We use the trivial upper bound $|P|$ for each point in $\{v_{j+1}, \ldots, v_{d}\}$. And once they are chosen the $(d+1)$-th point must lie in the unique $(d-1)$-flat containing $\{v_1, \ldots, v_d\}$, so there are at most $q^{d-1}$ choices for $v_{d+1}$. Combined with the fact $|P|\leq q^d$ and Lemma~\ref{critical_supersaturation_Fqd}, we can compute
    \begin{equation*}
        \Delta_j(\mathcal{H}) \le |P|^{d-j} q^{d-1} \leq \left(\frac{q^d}{|P|}\right)^{j-2}|P|^{d-j} q^{d-1}=\frac{1}{c_1}\cdot c_1\frac{|P|^d}{q}\cdot \l( \frac{q^{d-2}}{k^2} \r)^{j-1}\leq \frac{1}{c_1}\cdot \frac{|\mathcal{H}|}{|P|} \tau^{j-1}.
    \end{equation*}

    To get the upper bound of $\Delta_{d+1}(\mathcal{H})$, notice that $\Delta_{d+1}(\mathcal{H}) \leq 1$, so we have
    \begin{equation*}
        \Delta_{d+1}(\mathcal{H}) \leq 1
        \le \frac{1}{c_1} \cdot c_1\frac{|P|^d}{q} \left(\frac{1}{k q^{1-\frac{1}{d}}}\right)^{d}
        \le \frac{1}{c_1} \cdot \frac{|\mathcal{H}|}{|P|} \tau^{d}. \qedhere
    \end{equation*}
\end{proof}

\subsection{The third container theorem}
We prove the following container theorem for general position sets in $\mathbb{F}^d_q$ using Lemmas~\ref{critical_supersaturation_Fqd} and~\ref{refined_critical_supersaturation_Fqd}. 

\begin{theorem}\label{thm:container3}
    For every integer $d\ge 2$ there exists a constant $c>0$ such that for sufficiently large prime power $q$ and real number $k \in [4dq^{d-2}, q^{d-1-\frac{1}{d}}]$, we have a collection $\mathcal{C}=\mathcal{C}(q,k)$ of point sets (called containers) in $\mathbb{F}^d_q$ satisfying
    \begin{enumerate}[(i)]
        \item every general position set of $\mathbb{F}^d_q$ is contained in some member of $\mathcal{C}$, 
        \item $|C|\le kq$ for all $C\in \mathcal{C}$, and
        \item $|\mathcal{C}|\le \exp \l(ck^{-1}q^{d-1}(\log q)^2\r)$.
    \end{enumerate}
\end{theorem}

\begin{proof}
    Let $k \in [4dq^{d-2}, q^{d-1-\frac{1}{d}}]$ be fixed. We set $\mathcal{C}_{0} = \{\mathbb{F}_q^d\}$ and $\varepsilon = 1/e$. Suppose that $\mathcal{C}_j$ has been defined for some integer $j \ge 0$. We define $\mathcal{C}_{j}^{>}$ and $\mathcal{C}_{j}^{\le}$ in the same way as the proof of Theorem~\ref{thm:container2}. If $\mathcal{C}_{j}^{>} = \emptyset$, then we let $\mathcal{C} = \mathcal{C}_j$. Otherwise, for each $C\in \mathcal{C}_{j}^{>}$ we set
    \begin{equation*}
    k_C = {|C|}/{q},\quad
       \tau_{C} = \max\l\{k_C^{-1} q^{\frac{1}{d}-1},~ k_C^{-2}q^{d-2}\r\}, 
        \quad\text{and}\quad \hat{\tau}_C = c_1 \cdot 2^{\binom{d+1}{2}}\cdot 12(d+1)! \cdot \tau_C, 
    \end{equation*}
    where $c_1>0$ is the constant guaranteed by Lemma~\ref{refined_critical_supersaturation_Fqd}. Furthermore, let $\mathcal{H}_C = \mathcal{H}'_{q,d}[C]$ be the induced subgraph of $\mathcal{H}'_{q,d}$ on $C$, and recall that $\mathcal{H}'_{q,d}$ is the hypergraph defined by coplanar $(d+1)$-tuples in $\mathbb{F}_{q}^d$. Since $|C| > kq \ge 4dq^{d-1}$, it follows from Lemma~\ref{refined_critical_supersaturation_Fqd} that $\mathcal{H}_C$ is $(c_1, \tau_C)$-bounded. Similar to the proof of Theorem~\ref{thm:container2}, we can check that $\Delta\left(\mathcal{H}_C,\hat{\tau}_C\right) \leq \frac{\varepsilon}{12(d+1)!}$. In addition, since $q$ is sufficiently large, we have $\hat{\tau}_C \le \frac{1}{200\cdot (d+1) \cdot (d+1)!}$. It follows from Lemma~\ref{hypergraphcontainer}, applied to $\mathcal{H}_C$, that there exists a collection $\mathcal{C}_{C}$ of subsets of $C$ such that
    \begin{enumerate}[(a)]
        \item every independent set of $\mathcal{H}_C$ is contained in some member of $\mathcal{C}_C$,
        \item the size of $\mathcal{C}_C$ satisfies 
            \begin{align*}
                |\mathcal{C}_C|
                & \leq \exp\l( 1000\cdot (d+1) \cdot \left((d+1)!\right)^3 \cdot |C| \cdot \hat{\tau}_C\cdot \log\left(\frac{1}{\varepsilon}\right)\cdot \log\left(\frac{1}{\hat{\tau}_C}\right) \r) \\
                & \le \exp\l( c_2 \cdot \max\l\{\frac{1}{q^{1-\frac{1}{d}}},~\frac{q^{d-2}}{k_C}\r\} q \log q \r)\\
                & \le \exp\l( c_2 \cdot \max\l\{\frac{1}{q^{1-\frac{1}{d}}},~\frac{q^{d-2}}{k}\r\} q \log q \r)\\
                & \le \exp\l( c_2 k^{-1}q^{d-1}\log q \r),
            \end{align*} 
            where $c_2$ is a large constant depending on $d$ and the inequality follows from the definitions of $\hat{\tau}_C$ and $k_C$, and
        \item for every $C' \in \mathcal{C}_C$ the induced subgraph of $\mathcal{H}_C$ on $C'$ has size at most $\varepsilon|\mathcal{H}_C|$.
    \end{enumerate} Having constructed $\mathcal{C}_C$ for each $C \in \mathcal{C}_{j}^{>}$, we define $\mathcal{C}_{j+1}$ and proceed with the iteration in the same way as in the proofs of Theorems~\ref{thm:container1} and~\ref{thm:container2}.

    Finally, we can verify the statements (i) and (ii) claimed for our final container collection $\mathcal{C}$ similarly to the proof of Theorem~\ref{thm:container1}. Moreover, similar to the argument in the proof of Theorem~\ref{thm:container2}, using statements~(c) above and Lemma~\ref{critical_supersaturation_Fqd}, we can show that the iterative process above must stop after at most $d(d+1)\cdot\log q$ steps. Combining this fact with statement~(b) above, we obtain $|\mathcal{C}|\le \exp \l(ck^{-1}q^{d-1}(\log q)^2\r)$ with $c$ being a large constant depending on $d$. Hence we conclude the proof.
\end{proof}

\subsection{Proof of Theorem \ref{thm:randomturan2}}

We finish the proof of Theorem \ref{thm:randomturan2} in this subsection. First, we prove the following upper bounds for $\alpha(\mathbb{F}_q^d,p)$ with $p \leq q^{-d+\frac{2d}{d+1}+o(1)}$ using Theorem~\ref{thm:container3}.

\begin{proposition}\label{randomturan2_upper1}
For every integer $d\ge 2$, there exists a constant $c>0$ such that a.a.s.
\begin{equation*}
    \alpha(\mathbb{F}^d_q,p)
        \le
        \begin{cases}
            cq^{\frac{1}{d}}(\log q)^2, &\quad\text{if}\quad 0\le p\le q^{-d+\frac{2}{d}}(\log q)^2,\\
            cp^{\frac{1}{2}}q^{\frac{d}{2}}\log q, &\quad\text{if}\quad q^{-d+\frac{2}{d}}(\log q)^2\le p\le q^{-d+\frac{2d}{d+1}}(\log q)^2.
        \end{cases}
\end{equation*}
\end{proposition}
\begin{proof}
Let $\mathbf{S}_p$ be a $p$-random set of $\mathbb{F}_{q}^d$ and $\mathbf{X}_m(p)$ denote the number of general position sets of size $m$ in $\mathbf{S}_p$. Also let $c >0$ be the constant guaranteed by Theorem~\ref{thm:container3}. By enlarging $c$ if necessary, we can assume $c \ge 4e$ without loss of generality.

First we consider the range $q^{-d+\frac{2}{d}}(\log q)^2\le p\le q^{-d+\frac{2d}{d+1}}(\log q)^2$. We set $k = p^{-\frac{1}{2}} q^{\frac{d}{2}-1} \log q$ and $m = cpkq$. We can check that\begin{equation*}
    4dq^{d-2} \ll q^{d-2+\frac{1}{d+1}} \le k \le q^{d-1-\frac{1}{d}}.
\end{equation*} Let $\mathcal{C} = \mathcal{C}(q,k)$ be the container collection outputted by Theorem~\ref{thm:container3}. By the properties of $\mathcal{C}$ guaranteed by Theorem~\ref{thm:container3}, we can estimate
\begin{equation}\label{eq:randomturan2_upper1}
        \begin{aligned}
            \mathbb{E}[\mathbf{X}_m(p)] \le \sum_{C\in \mathcal{C}}\binom{|C|}{m} p^m
            \le |\mathcal{C}|\binom{kq}{m}p^m
            &\le \exp\left( c k^{-1}q^{d-1} (\log q)^2 \right) \left(\frac{ekq}{m}\right)^m p^m\\
            &\le \exp\l( c k^{-1}q^{d-1} (\log q)^2 + m \cdot \log\l(\frac{epkq}{m}\r) \r)\\
            &\le \exp\l( m+m \cdot \log\l(\frac{1}{4}\r)\r) \\
            &= \exp(-cp^{\frac{1}{2}}q^{\frac{d}{2}} \log q) \rightarrow 0 \quad \text{as}\quad q\rightarrow\infty.
        \end{aligned}
    \end{equation}
Same as in the proof of Proposition~\ref{randomturan1_upper}, we can use Markov's inequality to argue that a.a.s.\begin{equation*}
    \alpha(\mathbb{F}^d_q,p) \leq m = cp^{\frac{1}{2}}q^{\frac{d}{2}}\log q.
\end{equation*}
For $p\le q^{-d+\frac{2}{d}}(\log q)^2$, the argument is again similar to the proof of Proposition~\ref{randomturan1_upper}. Combining the ``nondecreasingness'' of $\mathbb{E}[\mathbf{X}_m(p)]$ and \eqref{eq:randomturan2_upper1}, we can argue that $\alpha(\mathbb{F}^d_q,p)<cq^{\frac{1}{d}} (\log q)^2$ a.a.s. in this range.
\end{proof}

Next, we use Theorem~\ref{thm:container2} to establish a different type of upper bounds for $\alpha(\mathbb{F}_q^d,p)$.
\begin{proposition}\label{randomturan2_upper2}
For every integer $d\ge 2$, there exists a constant $c>0$ such that a.a.s.
\begin{equation*}
    \alpha(\mathbb{F}^d_q,p)
    \le
    \begin{cases}
        cq^{\frac{d}{d+1}}(\log q)^2, &\quad\text{if}\quad p\le q^{-\frac{1}{d+1}}(\log q)^2,\\
        cpq, &\quad\text{if}\quad q^{-\frac{1}{d+1}}(\log q)^2\le p\le 1.
    \end{cases}
\end{equation*}
\end{proposition}
\begin{proof}
Let $\mathbf{S}_p$ be a $p$-random set of $\mathbb{F}_{q}^d$ and $\mathbf{X}_m(p)$ denote the number of general position sets of size $m$ in $\mathbf{S}_p$. Also let $c >0$ be the constant guaranteed by Theorem~\ref{thm:container2}. By enlarging $c$ if necessary, we can assume $c \ge 8e$.

First we consider the range $p\ge q^{-\frac{1}{d+1}}(\log q)^2$. We set $m = cpq$ and let $\mathcal{C} = \mathcal{C}(q,k)$ be the container collection outputted by Theorem~\ref{thm:container2} with $k=2$. By the properties of $\mathcal{C}$ guaranteed by Theorem~\ref{thm:container2}, we can estimate
\begin{equation*}
     \mathbb{E}[\textbf{X}_m]
    \le \sum_{C\in \mathcal{C}}\binom{|C|}{m}p^m
     \le |\mathcal{C}|\binom{2q}{m}p^m  
    \le \exp\l(cq^{\frac{d}{d+1}}(\log q)^2+m\cdot\log\l(\frac{2epq}{m}\r)\r)\leq \exp(cpq).
\end{equation*} Same as in the proof of Proposition~\ref{randomturan1_upper}, we can use Markov's inequality to argue that a.a.s.\begin{equation*}
    \alpha(\mathbb{F}^d_q,p) \leq m = cpq.
\end{equation*} Also similar to the proof of Proposition~\ref{randomturan1_upper}, our upper bound for the range $p\le q^{-\frac{1}{d+1}}(\log q)^2$ follows from the ``nondecreasingness'' of $\mathbb{E}[\mathbf{X}_m(p)]$ as a function of $p$.
\end{proof}

Finally, we present the proof of Theorem~\ref{thm:randomturan2}. Due to its similarity to the proof of Theorem~\ref{thm:randomturan1}, we only give a brief summary of our arguments.
\begin{proof}[Proof Sketch of Theorem~\ref{thm:randomturan2}]
    Let $\mathbf{S}_p$ be a $p$-random set of $\mathbb{F}_q^d$. First we consider the upper bounds. During the range $q^{-d + o(1)}\leq p \le q^{-d + \frac{1}{d} + o(1)}$, it follows from $\alpha(\mathbb{F}_q^d,p)\leq |\mathbf{S}_p|$ and Chernoff's bound. During the range $q^{-d + \frac{1}{d} + o(1)}\le p\le q^{-d+\frac{2d}{d+1}+o(1)}$, we just apply Proposition~\ref{randomturan2_upper1}. And during the range $q^{-d+\frac{2d}{d+1}+o(1)} \le p\leq 1$, we can apply Proposition~\ref{randomturan2_upper2} to obtain the claimed upper bound.

    Next we consider the lower bounds. During the range $q^{-d + o(1)}\leq p \le q^{-d + \frac{1}{d}+o(1)}$, we can check that the expected number of coplanar $(d+1)$-tuples in $\mathbf{S}_p$ is much smaller than the expected number of $|\mathbf{S}_p|$, and our lower bound follows from deleting a point from each coplanar $(d+1)$-tuple. During the range $q^{-d + \frac{1}{d}+o(1)} \le p \le q^{-\frac{1}{d}+o(1)}$, it follows from the ``nondecreasingness'' of $\alpha(\mathbb{F}^d_q,p)$ as a function of $p$. And during the range $q^{-\frac{1}{d}+o(1)} \le p\leq 1$, the lower bound comes from the intersection of $p$-random set and the moment curve $\textbf{Y} = \textbf{S}_p \cap \left\{(x, \ldots, x^{d}) :~ x\in \mathbb{F}_q\right\}$.
\end{proof}

\section{Concluding Remarks}\label{sec:remark}
\begin{itemize}
    \item The problem of determining $\alpha(\mathbb{F}^d_q,p)$ remains open for $d \geq 3$ in the range $q^{-d+2/d}\le p\le q^{-1/(d+1)}$. We believe that the lower bounds are closer to the truth. One obstacle in generalizing our proof for $d=2$ is the lack of optimal supersaturation when using the expander mixing lemma. The reason is that the intersection of hyperplanes could be as large as $q^{d-2}$, resulting in a large number of overcounting of coplanar $(d+1)$-tuples. This problem might be overcome with some extra ideas. As the first step towards resolving this problem, we suggest the following question for $d=3$.
    \begin{problem}
    Is it true that for any point set $P$ in $\mathbb{F}^3_q$ with size $1000q$, there is a collection $\mathcal{S}$ of coplanar $4$-tuples in $P$ such that
      \begin{equation*}
          |\mathcal{S}| =  \Omega\l(q^3\r),~\Delta_1(\mathcal{S}) = O\l(q^2\r),~\Delta_2(\mathcal{S})= O\l(q^{4/3}\r),~\text{and}~\Delta_3(\mathcal{S})= O\l(q^{2/3}\r).
      \end{equation*}
    \end{problem}

    \item More generally, we say a point set $S\subset \mathbb{F}_q^d$ is \textit{$(k,c)$-evasive} if every $k$-flat of $\mathbb{F}_q^d$ contains at most $c$ points in $S$. Notice that a general position set in $\mathbb{F}_q^d$ is simply a $(d-1,d)$-evasive set. Extremal problems for evasive sets are extensively studied (see e.g. \cite{dvir2012subspace, ben2014note, sudakov2022evasive, SZ}), but the asymptotic behaviour of many related extremal quantities are still unknown. It seems plausible to extend results in this paper to $(k,c)$-evasive sets in $\mathbb{F}_q^d$ for certain values of $d$ and $(k,c)$, and we hope to return to this topic in future work.
    
    \item  It seems possible that the assumption $|P| \ge 4d q^{d-1}$ in Lemma~\ref{critical_supersaturation_Fqd} can be relaxed to $|P| \gg q^{d-2}$, which (if true) would be tight up to some constant factor, since a $(d-2)$-flat has size $q^{d-2}$ but contains no critical coplanar $(d+1)$-tuples. As weak evidence for such a relaxation, we prove that every point set $P$ in $\mathbb{F}_q^d$ with $|P| \gg q^{d-2}$ contains at least one critical coplanar $(d+1)$-tuple. We also have the following theorem confirming our speculation for $d=3$. See our proofs in the Appendix.
    \begin{theorem}\label{critical_supersaturation_Fq3}
        There exists an absolute constant $c>0$ such that if $P$ is a point set of size at least $cq$ in $\mathbb{F}_q^3$, then $P$ contains at least $\Omega\left({|P|^4}/{q}\right)$ critical coplanar $4$-tuples.
    \end{theorem}
    \end{itemize}

\section*{Acknowledgement}
Ji Zeng wishes to thank Jonathan Tidor for stimulating discussions at 2023 UCSD Workshop on Ramsey Theory. This work was initiated while Ji Zeng was visiting SCMS at the kind invitation of Hehui Wu.

\bibliographystyle{abbrv}
\bibliography{main}

\appendix

\section{Extremal results for critical coplanar tuples}

We first prove a Tur{\'a}n-type result for critical coplanar $(d+1)$-tuples in $\mathbb{F}_q^d$.
\begin{theorem}
    For every integer $d\ge3$ there exists a constant $c>0$ such that if $P$ is a point set of size at least $cq^{d-2}$ in $\mathbb{F}^d_q$, then there exists a critical coplanar $(d+1)$-tuple contained in $P$.
\end{theorem}
\begin{proof}
    We prove it by induction on $d$. First we consider the $d=3$ case. Since $\mathbb{F}_q^3$ can be covered by $q$ parallel hyperplanes, there exists one hyperplane $F_2$ such that $|F_2 \cap P| \geq 4$ (assuming $c$ is sufficiently large). If these $4$ points do not form a critical coplanar $4$-tuple, there exists a line $F_1 \subset F_2$ containing at least $3$ vertices of $P$. Notice that there are $q+1$ hyperplanes containing $F_1$. Hence by the pigeonhole principle, there is a hyperplane $F_2'$ containing $F_1$ and \begin{equation*}
        |F_2' \cap (P\setminus F_1)| \geq \frac{|P\setminus F_1|}{q+1} \geq \frac{cq-q}{q+1} \geq 2.
    \end{equation*} Hence, let $u,v$ be two points in $F_2' \cap (P\setminus F_1)$ and $\ell_{uv}$ be the line determined by them. By the fact that $|F_1 \cap P|\ge 3$, we have $|(F_1 \cap P)\setminus \ell_{uv}| \ge 2$. Then let $w,z$ be two points in $(F_1 \cap P)\setminus \ell_{uv}$, and notice that $\{u,v,w,z\}$ is a critical coplanar $4$-tuple (there are no collinear triples among them).

    Now we assume $d\geq 4$. Since $\mathbb{F}_q^d$ can be covered by $q$ parallel hyperplanes, there exists one hyperplane $F_{d-1}$ such that $|F_{d-1} \cap P| \geq cq^{d-3}$. By the inductive hypothesis, there exists a $d$-tuple $T$ that is critical coplanar inside the $(d-1)$-dimensional affine space $F_{d-1}$. As a consequence, $T$ spans a $(d-2)$-flat $F_{d-2} \subset F_{d-1}$. Notice that there are $q+1$ hyperplanes in $\mathbb{F}_q^d$ containing $F_{d-2}$. So there is a hyperplane $F_{d-1}'$ containing $F_{d-2}$ and \begin{equation*}
        |F_{d-1}' \cap (P \setminus F_{d-2})| \geq \frac{|P \setminus F_{d-2}|}{q + 1} \geq \frac{cq^{d-2}-q^{d-2}}{q+1} \geq \frac{c}{2} q^{d-3}.
    \end{equation*}
    
    Write $S = F_{d-1}' \cap (P \setminus F_{d-2})$ and fix a point $u\in S$. For each point $v \in S\setminus \{u\}$, if the union of either $(d-1)$-tuple in $T$ and $\{u,v\}$ is not a critical coplanar $(d+1)$-tuple in $\mathbb{F}_q^d$, then we say that $v$ is a \textit{bad vertex}. It suffices for us to show that not all of $S\setminus \{u\}$ are bad vertices. To do this, we need the following claim.
    \begin{claim}\label{claim:critical_turan}
        Suppose $v \in S\setminus \{u\}$ is a bad vertex, then any $(d-1)$-tuple in $T$ contains a $(d-2)$-tuple $\tau$ such that $\tau \cup\{u,v\}$ lies on a $(d-2)$-flat.
     \end{claim}
     \begin{proof}
         Fix a $(d-1)$-tuple $T' \subset T$. By definition, $T'\cup\{u,v\}$ is not a critical coplanar $(d+1)$-tuple. So $T'\cup\{u,v\}$ contains a $d$-tuple $\tau'$ that is affinely dependent. Since $T$ is critical coplanar inside the $(d-1)$-dimensional affine space $F_{d-1}$, the affine span of $T'$ coincides with the affine span of $T$, which is $F_{d-2}$. Meanwhile, $F_{d-2}$ does not contain $u$ or $v$. So the affine span of $T' \cup \{u\}$ or $T' \cup \{v\}$ has dimension $d-1$. As a consequence, $\tau'$ cannot be either $T' \cup \{u\}$ or $T' \cup \{v\}$. Hence, there exists a $(d-2)$-tuple $\tau$ in $T'$ such that $\tau'=\tau\cup\{u,v\}$, and we know that $\tau'$ lies on a $(d-2)$-flat.
     \end{proof}
     
     Now let $v$ be a bad vertex and $\ell_{uv}$ be the line determined by $u,v$. By Claim~\ref{claim:critical_turan}, there exists a $(d-2)$-tuple $\tau$ such that $\tau \cup \{u,v\}$ lies on a $(d-2)$-flat $F'_{d-2}$. Next, we choose a $(d-1)$-tuple $T' \subset T$ such that $\tau \not \subset T'$. By Claim~\ref{claim:critical_turan}, there exists a $(d-2)$-tuple $\tau' \subset T'$ such that $\tau'\cup \{u,v\}$ lies on a $(d-2)$-flat $F''_{d-2}$. We claim that $F'_{d-2} \neq F''_{d-2}$. Indeed, since $|\tau \cup \tau'|\ge d-1$ and $T$ is critical planar in $F_{d-1}$, the span of $\tau \cup \tau'$ coincides with the span of $T$, which is $F_{d-2}$. If $F'_{d-2} = F''_{d-2}$, then we can conclude $F_{d-2} = F'_{d-2}$, which contradicts to the fact $u,v \not\in F_{d-2}$. Thus $v$ lies on a $(d-3)$-flat $F'_{d-2} \cap F''_{d-2}$ and this $(d-3)$-flat only depends on $\tau, \tau' \subset T$. Since there are $\binom{d}{d-2}$ choices for a $(d-2)$-tuple in $T$, it follows that there are at most $\binom{d}{d-2}^2 q^{d-3}$ possible locations for a bad vertex $v$. Recall that $|S| \ge \frac{c}{2}q^{d-3}$. Hence, we can find $v\in S\setminus \{u\}$ which is not a bad vertex given that $c$ is sufficiently large.
\end{proof}

Next we prove Theorem~\ref{critical_supersaturation_Fq3}, the ``asymptotically best possible'' supersaturation result for critical coplanar $4$-tuples in $\mathbb{F}_q^3$. We restate the theorem here.
\begin{theorem}
    There exists an absolute constant $c>0$ such that if $P$ is a point set of size at least $cq$ in $\mathbb{F}_q^3$, then $P$ contains at least $\Omega\left({|P|^4}/{q}\right)$ critical coplanar $4$-tuples.
\end{theorem}
\begin{proof}
Fix an arbitrary pair $\{u,v\}\subset P$ and let $\ell$ be the line determined by $u,v$. Notice that there are $q+1$ planes $\pi_1,\dots,\pi_{q+1}$ containing $\ell$. First, we discard all $\pi_i$ such that $\pi_i\setminus \ell$ contains less than $|P|/(10q)$ points from $P$. Notice that at least $|P|/10$ points are not in any discarded $\pi_i$ (here we assume $c$ is sufficiently large and use $|P| \ge cq$). Next, we label each remaining hyperplane $\pi_i$ as follows: if there exists a line $\ell_i$ such that $u \in \ell_i\subset \pi_i$ and at least a $1/3$-fraction of points from $(\pi_i\setminus \ell)\cap P$ lie in $\ell_i$, then we label $\pi_i$ with $u$; if there exists a line $\ell_i$ such that $v \in \ell_i\subset \pi_i$ and at least a $1/3$-fraction of points from $(\pi_i\setminus \ell)\cap P$ lie in $\ell_i$, then we label $\pi_i$ with $v$; otherwise we label $\pi_i$ with $*$. Then, by the pigeonhole principle, there are planes $\pi_1,\dots,\pi_{k}$ sharing the same label such that their union contains at least $|P|/30$ points from $P\setminus \ell$. Write $x_i=|(\pi_i\setminus \ell)\cap P|$ for $i \in [k]$, then we have\begin{equation*}
    \sum_{i=1}^{k} x_i \geq \frac{|P|}{30} \text{\quad and \quad} k \leq q+1.
\end{equation*} 

If the planes $\pi_1,\pi_2,\dots,\pi_{k}$ are all labeled by $*$, then there are $x_i^2/3$ critical coplanar $4$-tuples containing $\{u,v\}$ on each $\pi_i$. Indeed, we first choose a point $w \in \pi_i \setminus \ell$ out of $x_i$ choices. Then we discard all the points on the lines $\overline{uw}$ and $\overline{vw}$. Because the plane $\pi_i$ is labelled by $*$, there are at least $x_i/3$ points in $P\cap \pi_i$ remaining, so we can choose any $z$ from them and $\{u,v,w,z\}$ will be critical coplanar. Hence, we have at least $\sum_{i=1}^{k} x_i^2/3$ many critical coplanar $4$-tuples. By Jensen's inequality, there are at least $\Omega(|P|^{2}/q)$ critical coplanar $4$-tuples containing $\{u,v\}$. In this particular case, we say that $\{u,v\}$ is a \textit{good pair}.

If the planes $\pi_1,\pi_2,\dots,\pi_{k}$ are all labeled by $u$ (for example), we say that $u$ is the \textit{pivot} of the \textit{bad pair} $\{u,v\}$. Notice that for $1\leq i<j\leq k$, there are $\binom{x_i/3}{2}\binom{x_j/3}{2}$ many critical coplanar $4$-tuples $\{w_1,w_2,w_3,w_4\}$ such that the two lines $\overline{w_1w_2}$ and $\overline{w_3w_4}$ intersect at $u$. Indeed, we can just choose $w_1,w_2$ from $\ell_i\cap P$ and $w_3,w_4$ from $\ell_j\cap P$. In this way, the number of such $4$-tuples we produced is at least, up to a constant factor, \begin{equation}
    \sum_{1\leq i<j\leq k} x_i^2x_j^2.\label{eq:critical_supersaturation_Fq3}
\end{equation}

The following claim, whose proof we postpone, optimizes this quantity in terms of $|P|$.
\begin{claim}\label{claim:critical_supersaturation_Fq3}
    When $\{u,v\}$ is a bad pair, the quantity \eqref{eq:critical_supersaturation_Fq3} is at least $\Omega(|P|^{3}/q)$.
\end{claim}

We consider the above process for all pairs $u,v\in P$. There are $\Omega(|P|^2)$ pairs and each one is either good or bad. If there are $\Omega(|P|^2)$ many good pairs, we produce $\Omega((|P|^2) \cdot (|P|^{2}/q))$ critical coplanar $4$-tuples, where each distinct tuple is repeated at most $\binom{4}{2}$ times. Indeed, for each critical coplanar $4$-tuple, there are $\binom{4}{2}$ ways to determine $\{u,v\}$.

If there are $\Omega(|P|^2)$ many bad pairs, we produce $\Omega((|P|^2) \cdot (|P|^{3}/q))$ critical coplanar $4$-tuples, where each distinct tuple is repeated at most $3|P|$ times. Indeed, for each critical coplanar $4$-tuple, there are $3$ candidates for the pivot. Once the pivot is determined, there are at most $|P|$ choices for the bad pair $\{u,v\}$.

Overall, we have counted $\Omega(|P|^4/q)$ critical coplanar $4$-tuples in $P$. Hence it suffices for us to justify Claim~\ref{claim:critical_supersaturation_Fq3}. Without loss of generality, we assume $x_1\geq x_2\geq \dots \geq x_{k}$. We check that, fixing $x_1$, the quantity \eqref{eq:critical_supersaturation_Fq3}, considered as a function of $x_2,\dots,x_{k}$, is Schur-convex using the Schur-Ostrowski criterion (see \cite{schur1923uber,ostrowski1952quelques}).
    
For any $2\leq i\neq j \leq k$, we notice that\begin{equation*}
        f(x_2,x_3,\dots,x_{k}):= \eqref{eq:critical_supersaturation_Fq3} = x_i^2x_j^2 + (x_i^2+x_j^2)\left(\sum_{\substack{\iota \neq i,j\\ 1\leq \iota\leq k}} x_\iota^2 \right) + \text{other terms independent of $x_i$ and $x_j$}.
    \end{equation*} Then we can compute\begin{equation*}
        (x_i-x_j)\left( \frac{\partial f}{\partial x_i} - \frac{\partial f}{\partial x_j}\right) = 2(x_i-x_j)^2\left(\sum_{\substack{\iota \neq i,j\\ 1\leq \iota\leq k}} x_\iota^2 - x_ix_j\right) \geq 0,
    \end{equation*}where the last inequality follows from $x_1^2\geq x_ix_j$. Hence $f$ satisfies the Schur-Ostrowski criterion.

    Because we have the hypothesis that $\{u,v\}$ is a bad pair, we have a special condition $x_1/3 \leq q$ since at least a $1/3$-fraction of points from $(\pi_i\setminus \ell)\cap P$ is contained in a line. So we have $\sum_{i=2}^{k} x_i \geq |P|/100$ (here we use $|P| \ge cq$). Therefore, by Schur-convexity, the quantity \eqref{eq:critical_supersaturation_Fq3} is lower bounded by when $x_i$'s (with $2\leq i\leq k$) are roughly equal, which means\begin{equation*}
        \sum_{2\leq i<j\leq k} x_i^2x_j^2 \geq \Omega\l(\l(\frac{|P|}{k}\r)^4(k)^2\r) \geq \Omega(|P|^3/q).
    \end{equation*} Here we used $|P| \ge cq$ and $k \le q+1$ for the last inequality. This proves Claim~\ref{claim:critical_supersaturation_Fq3} and concludes our proof.
\end{proof}
\end{document}